\documentclass[a4paper]{amsart}
\usepackage{graphicx}
\usepackage{amscd}
\usepackage{amsmath}
\usepackage{amsfonts}
\usepackage{amssymb}
\usepackage{setspace}
\usepackage{enumerate}         
\usepackage{fixme}
\usepackage{color}
\usepackage{url}
\usepackage{amsthm}
\usepackage{bm}
\usepackage{xy}
\usepackage{enumitem}
\usepackage{tikz-cd}
\usetikzlibrary{cd}

\theoremstyle{plain}
\newtheorem{theorem}{Theorem}[section]

\newtheorem{corollary}[theorem]{Corollary}

\newtheorem{lemma}[theorem]{Lemma}

\newtheorem{proposition}[theorem]{Proposition}

\newtheorem{definition}[theorem]{Definition}

\newtheorem{assumption}[theorem]{Assumption}
\theoremstyle{remark}
\newtheorem{remark}[theorem]{Remark}

\newtheorem{example}[theorem]{Example}

\numberwithin{equation}{section}

\newcommand{\ind}{1\!\kern-1pt \mathrm{I}}
\newcommand{\rsto}{]\!\kern-1.8pt ]}
\newcommand{\lsto}{[\!\kern-1.7pt [}

\renewcommand{\theenumi}{\rm{(\roman{enumi})}}
\numberwithin{equation}{section}

\newcommand{\diag}{\mathop{\mathsf{diag}}}

\renewcommand{\rho}{\varrho}

\begin{document}
\title[Generalized Feller processes and Markovian lifts: the affine
case]{Generalized Feller processes and Markovian lifts of stochastic
  Volterra processes: the affine case}

\begin{abstract}
  We consider stochastic (partial) differential equations appearing as
  Markovian lifts of affine Volterra processes with jumps from the
  point of view of the generalized Feller property which was
  introduced in e.g.~\cite{doetei:10}. In particular we provide new
  existence, uniqueness and approximation results for Markovian lifts
  of affine rough volatility models of general jump diffusion type. We
  demonstrate that in this Markovian light the theory of stochastic
  Volterra processes becomes almost classical.
\end{abstract}

\thanks{The authors are grateful for the support of the ETH
  Foundation. Christa Cuchiero thanks the Forschungsinstitut f\"ur
  Mathematik for its support of a research stay at ETH Z\"urich in
  fall 2016. She additionally gratefully acknowledges financial
  support by the Vienna Science and Technology Fund (WWTF) under grant
  MA16-021. \\ Both authors are very grateful to Sergio Pulido for
  helping us to improve Subsection 5.2 and to an anonymous referee
  whose comments helped to improve the article along several lines.}


\author{Christa Cuchiero and Josef Teichmann}

\address{Vienna University of Economics and Business, Welthandelsplatz
  1, A-1020 Vienna and ETH Z\"urich, R\"amistrasse 101, CH-8092
  Z\"urich}
\maketitle

\section{Introduction}\label{sec:intro}

In the realm of the recent discovery of rough volatility models (see
e.g., \cite{AY:14, GatJaiRos:14, ER:16, EFR:16}), stochastic Volterra
processes have been studied extensively, for instance in
\cite{JaiRos:16, AbiLarPul:17, ElERos:17} and the references
therein. It is well known that Markovian lifts of such processes
exist, but it has rarely been detailed how these lifts look like in
general, with the notable exception of \cite{MS:15} and the recent
article \cite{AbiElE:18b}.  If at all, these lifts have been
constructed by solving first the stochastic Volterra equations and
identifying a state space a posteriori, on which respective Markov
properties show up. Neither maximum principles, nor strongly
continuous semigroups, nor Feller properties, nor approximation
properties, nor Kolmogorov backward equations have been considered in
this context so far. The reason is of course that one enters the world
of SPDEs where each of the previously mentioned concepts needs
particular care and is often cumbersome due to the lack of local
compactness of the underlying state spaces. Also the question whether
all those considerations are restricted to Brownian driven processes
has not yet been investigated due to technical difficulties in writing
down (probabilistically) weak or martingale solution concepts with
jump drivers that are not standard Poisson random measures.
 
With this article we aim to bridge this gap and introduce a functional
analytic setting which clearly shows that it is an advantage to
consider Markovian lifts right from the beginning without solving the
Volterra equation in the first place. As a particular important
example we deal with Volterra processes whose kernels are Laplace
transforms of (signed) measures (see Section
\ref{sec:markovianlift}). Their Markovian lifts can be represented in
terms of signed measure valued processes which can be treated in a
systematic way comparable to Feller processes on locally compact state
spaces. In particular, we demonstrate that many results in the realm
of stochastic Volterra processes appear transparent and clearly
structured in this framework. Indeed, thanks to the generalized Feller
theory introduced in \cite{doetei:10} many properties and theorems can
be formulated abstractly while in the literature so far often rather
concrete and example-like specifications have been considered. The
passage to these \emph{generalized} Feller processes is necessary,
since no \emph{standard} Feller formulation is available due to the
lack of local compactness, as demonstrated in Example
\ref{ex:local_compactness}. This is in contrast to superprocesses,
see, e.g.~\cite{daw:12}, which take values in spaces of nonnegative
measures which are locally compact when the underlying space is
compact.

We do also believe that the presented theory of generalized Feller
semigroups and processes is a natural framework for analyzing SPDEs
valid also far beyond Markovian lifts of stochastic Volterra
processes: it opens the door to treat large deviations, long term
behavior, numerical procedures, or asymptotic results in a general and
simple framework.

This first article takes as guiding example affine Volterra processes
on $\mathbb{R}_+$ of the form
\begin{align}\label{eq:V}
  V_t = h(t) + \int_0^t K(t-s) dX_s,
\end{align}
where $h$ is a deterministic function, $K$ a deterministic kernel and
$X$ a semimartingale whose characteristics depend linearly on the
state of the Volterra process. In a subsequent article the general
(non-affine) case, including properly formulated martingale problems
and strong existence results, as well as other versions of Markovian
lifts will be considered. Here we focus on Markovian lifts where the
kernel $K$ can be represented as
\begin{align}\label{eq:kernelintro}
  K(t)= \langle g, \mathcal{S}^*_t \nu  \rangle
\end{align}
with $(\mathcal{S}^*_t)_{t \geq 0}$ a strongly continuous semigroup
acting on a Banach space $Y^*$, $\nu \in Y^*$ (or in a slightly bigger
space), $g \in Y$ with $Y$ the pre-dual of $Y^*$ and pairing
$\langle \cdot, \cdot \rangle$. This abstract setting includes the
above mentioned signed measure valued lifts as well as what we call
``forward process lifts'', where a variant of the latter has also been
considered in \cite{AbiElE:18b} in the special case of the Volterra
Heston model. This richness already indicates that our almost
axiomatic approach does not only simplify arguments but also unifies
several concepts and branches of the literature. In particular, it
provides due to its simple structure new existence and uniqueness
results for \eqref{eq:V}, and leads to so far unknown approximation
schemes of any order. For instance it will be easy to construct higher
order weak approximation schemes like Ninomiya-Victoir schemes with
precise and optimal convergence rates.

Inspired from Hawkes processes, see e.g.~\cite{hawkes:1971,
  boumezoued:2016}, and in contrast to some recent literature,
e.g. \cite{MS:15, AbiLarPul:17, AbiElE:18a, AbiElE:18b}, we do not
consider the Brownian motion driven case as the simplest one but
rather the case when the stochastic driver is a finite activity jump
process (as it is often true and useful in the theory of stochastic
processes): Brownian driven stochastic Volterra equations or
stochastic Volterra equations with more complicated jump structures
are then easy and well described limits (in the sense of generalized
Feller processes) of processes with finitely many jumps on compact
time intervals, from which then weak solutions can be
constructed. Another difference to most of the literature is that we
do not transform the processes in question into semimartingales to use
stochastic calculus but we rather work directly with their Markovian
structure. We also want to point out that with representation
\eqref{eq:kernelintro} we go well beyond standard assumptions on
Volterra kernels like complete monotonicity, in particular we do not
need resolvents of the first kind which do not always exist.

Let us outline in the sequel one of our lines of ideas in the above
addressed case of kernels that are Laplace transforms of signed
measures: we consider the vector space of finite, signed Borel
measures $Y^*:= \mathcal{M}(\mathbb{R}_{+} \cup \{\infty\})$ with its
weak-$*$-topology, see \cite{SchWol:99} or any other text book on
functional analysis, and look at the following simple linear
homogenous equation thereon
\[
  d \lambda_t (dx) = - x \lambda_t(dx) dt - w \nu (dx)
  \lambda_t(\mathbb{R}_{+} \cup \{\infty\}) dt,
\]
where the initial value $ \lambda_0 $ is a signed measure, $w \geq 0$
a constant and $ \nu $ a signed measure such that its Laplace
transform
\[
  \int_0^\infty \exp(-tx) \nu(dx) \geq 0
\]
is finite and non-negative for all $ t > 0 $, and belongs to
$ L^2_{\operatorname{loc}}(\mathbb{R}_{+}, \mathbb{R}_+) $. The
representation of the kernel in \eqref{eq:kernelintro} is here
\begin{align}\label{eq:kernelintro1}
  K(t)=\langle g, \mathcal{S}^*_t \nu \rangle= \langle1, \exp(-t \cdot) \nu \rangle= \int_0^\infty \exp(-tx) \nu(dx)
\end{align}
i.e.~the evaluation of the element $ \exp(-t \cdot) \nu $ in the dual
space $\mathcal{M}(\mathbb{R}_+ \cup \{\infty\})$ of
$Y= C_b(\mathbb{R}_{+} \cup \{ \infty \}, \mathbb{R}) $ on the
constant function $ 1 $.  Under these circumstances one can easily
prove two facts: first the solution defines a generalized Feller
process (see \cite{doetei:10} and the subsequent sections for details
on this notion) on the state space of all signed measures, and,
second, there is a closed invariant subspace $ \mathcal{E} $
consisting of initial measures $ \lambda_0 $ where the total mass
$\langle 1 , \lambda_t \rangle= \lambda_t(\mathbb{R}_{+} \cup
\{\infty\}) $ remains non-negative along its evolution in time for all
$w \geq 0$. Of course the total mass satisfies a one-dimensional
linear Volterra equation, whence the set $ \mathcal{E}$ of measures
$ \lambda_0 $ can be characterized by resolvents, see Section
\ref{sec:markovianlift_abstract} and Section \ref{sec:markovianlift}
for all details. Next we construct on this very state space
$ \mathcal{E} $ all sorts of jump diffusion processes with general
characteristics by Poisson approximations applying an approximation
theorem for generalized Feller semigroups (see Theorem
\ref{thm:approximation_gen}). This is parallel to classical work on
affine processes, see \cite{DufFilSch:03, CFMT:11, CKMT:14}, where
general existence is obtained by pure jump approximations. We can
finally show that the SPDE
\begin{align}\label{eq:SPDEintro}
  d \lambda_t (dx) = - x \lambda_t(dx) dt + \nu (dx) dX_t
\end{align}
for all semimartingales $X$ whose characteristics depend linearly on
$\lambda_t(\mathbb{R}_+ \cup \{\infty\})$ admits an analytically mild
and probabilistically weak solution. Moreover, the associated Cauchy
problem has a solution given in terms of a generalized Feller
semigroup acting on a well specified set of functions.  In particular,
by the variation of constants formula and \eqref{eq:kernelintro1} it
is easily seen that the total mass
$\langle 1 , \lambda_t \rangle= \lambda_t(\mathbb{R}_{+} \cup
\{\infty\}) $ solves \eqref{eq:V} with
$h(t)=\int_0^{\infty} \exp(-xt) \lambda_0(dx)$.

We also want to point out that we obtain generic path properties,
i.e.~in this general setting we shall always have c\`ag versions but
not necessarily c\`adl\`ag versions. This is in contrast to Feller
processes where c\`adl\`ag versions do always exist. This is due to the
fact that generalized Feller processes are not always
quasi-contractive.

This previous example highlights our results which can be summarized
as follows:
\begin{itemize}
\item we provide a full solution theory for univariate stochastic
  Volterra equations \eqref{eq:V} driven by jump-diffusions. In
  particular for a large class of functionals we can provide solutions
  of the respective Kolmogorov equations.
\item we provide approximation schemes, (generalized) maximum
  principles (in certain cases), (probabilistically) weak solutions
  concepts for Markovian lifts of jump-diffusion driven stochastic
  Volterra equations \eqref{eq:V}.
\item we provide a solution theory of the corresponding generalized
  Riccati differential equations on the pre-dual space
  $ \mathcal{E}_* $ and the respective affine transform formulas.
\item we can go beyond standard assumptions on kernels of stochastic
  Volterra equations such as complete monotonicity.
\item we provide an abstract theory of Markovian lifts which is not
  restricted to the above introduced univariate case. Indeed, by
  considering as $Y^*$ the space of $\mathbb{R}^d$-valued measures on
  $E$ for some locally compact space $E$, \eqref{eq:SPDEintro} can be
  generalized to the following multivariate case
  \[
    d \lambda^l_t (dx) = - f^l(x) \lambda^l_t(dx) dt + \sum_{j=1}^d
    \nu^l_{j}(dx) dX^j_t, \quad
  \]
  for $l \in \{1, \ldots,d\}$ in a straightforward way. This is due to
  the completely general character of Proposition
  \ref{prop:jump_perturbation} below.  Here, $ f^l$ are functions from
  $ E \to \mathbb{R}_{+} $ satisfying appropriate conditions, and
  $ X^j $ are semimartingales whose characteristics depend linearly on
  $\lambda$.  Note that this allows for Volterra process
  specifications of the form
  \[
    V_t^l=h^l(t)+ \int_0^t \sum_{j=1}^d K^l_j(t-s) dX^j_s,
  \]
  for kernels $K^l_j$ given by
  $K^l_j(t)= \int_E g^l(x) e^{-t f^l(x)}\nu^l_{j}(dx)$ for some
  bounded continuous $g^l \in C_b( E ; \mathbb{R})$ and
  $h^l(t)= \int_E g^l(x) e^{-t f^l(x)}\lambda^l_{0}(dx)$.

\item Heston-like or Barndorff-Nielsen-Shephard-like structures can be
  analogously constructed again by choosing appropriate Banach spaces
  $ Y $. For instance the rough Heston case \cite{ER:16} can be
  treated with
  $ Y = C_b(\mathbb{R}_{+} \cup \{\infty\}; \mathbb{R}^2) $ and the
  equations
  \begin{align*}
    d \lambda^1_t(dx)& = -x \lambda^1_t(dx) dt + \nu(dx) \sqrt{\langle
                       1, \lambda^1_t \rangle} dB^1_t,\\
    d \lambda^2_t(dx) &= -x \lambda^2_t(dx) dt + \delta_0(dx)
                        \sqrt{\langle 1, \lambda^1_t \rangle} dB^2_t,
  \end{align*}
  where $B^1$ and $B^2$ are correlated Brownian motions with
  $ d \langle B^1, B^2 \rangle_t = \rho dt $.  Note that the variance
  process is given by $\lambda_t^1(\mathbb{R}_+ \cup \{\infty\})$ and
  the log price process by
  $\lambda_t^2(\mathbb{R}_+ \cup \{\infty\})$.
\item Covariance matrix (or cone-) valued processes can be constructed
  by considering positive semidefinite symmetric matrix (or cone-)
  valued measures as respective state space. Here particular geometric
  restrictions appear in the diffusion driven case.  Markovian lifts
  of Volterra type processes on this more involved state space,
  relevant for rough covariance modeling, are treated in \cite{CT:19}.
\end{itemize}

The remainder of the article is as follows: in Section \ref{sec:not}
we introduce some notation and review certain functional analytic
concepts. In Section \ref{sec:genFeller}, we deal with generalized
Feller processes as introduced in \cite{doetei:10}. In Section
\ref{sec:approximationTheorems} we show simple approximation theorems
for generalized Feller semigroups beyond standard Trotter-Kato type
theorems, and we prove a crucial existence result for pure jump
processes with finite but unbounded intensity. In Sections
\ref{sec:markovianlift_abstract}, \ref{sec:markovianlift} and
\ref{sec:markovianlift_forward} we apply the presented theory to SPDEs
which are lifts of affine Volterra processes.

\subsection{Notation and some functional analytic
  notions}\label{sec:not}

For the background in functional analysis we refer to the excellent
textbook \cite{SchWol:99} as main reference and to the equally
excellent books \cite{engnag:00,paz:83} for the background in strongly
continuous semigroups. We emphasize however that only very basic
knowledge in functional analysis and strongly continuous semigroups is
required.

We shall often apply the following notations: let $ Y $ be a Banach
space and $ Y^* $ its dual space, i.e.~the space of linear continuous
functionals with the strong dual norm
\[ {\|\lambda\|}_{Y^*} = \sup_{\|y\| \leq 1} | \langle y , \lambda
  \rangle | \, ,
\]
where $ \langle y , \lambda \rangle := \lambda(y) $ denotes the
evaluation of the linear functional $ \lambda $ at the point
$ y \in Y $. Since cones $ \mathcal{E} $ of $ Y^* $ will be our
statespaces, we denote the polar cones in pre-dual notation, i.e.
\[
  \mathcal{E}_* = \big \{ y \in Y \, | \; \langle y , \lambda \rangle
  \leq 0 \text{ for all } \lambda \in \mathcal{E} \big \} .
\]
We denote spaces of bounded linear operators from Banach spaces
$ Y_1 $ to $ Y_2 $ by $ L(Y_1,Y_2) $ with norm
\[ {\| A \|}_{L(Y_1,Y_2)} := \sup_{{\|y_1\|}_{Y_1}\leq 1} {\| Ay_1
    \|}_{Y_2} \, .
\]
If $Y_1=Y_2$ we only write $\|\cdot \|_{L(Y_1)}$. On $ Y^* $ we shall
usually consider beside the strong topology (induced by the strong
dual norm) the weak-$*$-topology, which is the weakest locally convex
topology making \emph{all} linear functionals
$ \langle y , \cdot \rangle $ on $ Y^* $ continuous.  Let us recall
the following facts:
\begin{itemize}
\item The weak-$*$-topology is metrizable if and only if $ Y $ is
  finite dimensional: this is due to Baire's category theorem since
  $Y^*$ can be written as a countable union of closed sets, whence at
  least one has to contain an open set, which in turn means that
  compact neighborhoods exist, i.e.~a strictly finite dimensional
  phenomenon.
\item Norm balls $ K_R $ of any radius $ R $ in $ Y^* $ are compact
  with respect to the weak-$*$-topology, which is the Banach-Alaoglu
  theorem.
\item These balls are metrizable if and only if $ Y $ is separable:
  this is true since $ Y $ can be isometrically embedded into
  $ C(K_1) $, where $ y \mapsto \langle y,\cdot \rangle $, for
  $ y \in Y $. Since $ Y $ is separable, its embedded image is
  separable, too, which means -- by looking at the algebra generated
  by $ Y $ in $ C(K_1) $ -- that $ C(K_1) $ is separable, which is the
  case if and only if $ K_1 $ is metrizable.
\end{itemize}
Since we do not need metrizability of $K_R$, we do \emph{not} assume
that our Banach spaces $ Y $ are separable.

A family of linear operators $ {(P_t)}_{t \geq 0} $ on a Banach space
$ Y $ with $ P_t P_s = P_{t+s} $ for $ s,t \geq 0 $ and with
$ P_0 =I $ where $I$ denotes the identity is called strongly
continuous semigroup if $ \lim_{t \to 0} P_t y = y $ holds true for
every $ y \in Y $. We denote its generator usually by $ A $ which is
defined as $ \lim_{t \to 0} \frac{P_t y - y}{t} $ for all
$ y \in \operatorname{dom}(A) $, i.e.~the set of elements where the
limit exists. Notice that $ \operatorname{dom}(A) $ is left invariant
by the semigroup $ P $ and that its restriction on the domain equipped
with the operator norm
\[ {\| y \|}_{\operatorname{dom}(A)} := \sqrt{\|y\|^2 + \|Ay\|^2}
\]
is again a strongly continuous semigroup.

Moreover, as already used in the introduction, $\mathcal{M}(E)$
denotes the space of signed finite measures on $E$ equipped with the
Borel $\sigma$-algebra and $\mathcal{M}_+(E)$ the space of finite
nonnegative measures.

\section{Generalized Feller semigroups and
  processes}\label{sec:genFeller}

Feller semigroups have proved to be useful in the context of locally
compact state spaces. Generalized Feller semigroups serve the same
purpose on state spaces which are not locally compact, which is a
typical infinite dimensional phenomenon.  Local compactness is
replaced by adjoining a \emph{proper} weight function to a state space
$ X $, which measures explosion. Interestingly even in the locally
compact case generalized Feller processes yield new results relevant
for the theory of affine processes, see e.g.~\cite{DufFilSch:03,
  CFMT:11, CKMT:14}, since unbounded test functions can be
considered. We refer to \cite{doetei:10} for all necessary details and
proofs, and also to the references therein. Notice that even in the
infinite dimensional setting there are notable examples of local
compactness, see Remark \ref{rem:local_compactness}, but we point out
that our cases are not locally compact in their natural topology, see
Example \ref{ex:local_compactness}.

\subsection{Defintions and results}
First we introduce weighted spaces and state a central Riesz
representation result. The underlying space $X$ here is a completely
regular Hausdorff topological space.
\begin{definition}
  A function $\rho\colon X\to(0,\infty)$ is called \emph{admissible
    weight function} if the sets
  $K_R:=\left\{ x\in X\colon \rho(x)\le R \right\}$ are compact for
  all $R>0$.
\end{definition}
An admissible weight function $\rho$ is necessarily lower
semicontinuous and bounded from below by a positive constant. We call
the pair $X$ together with an admissible weight function $\rho$ a
\emph{weighted space}. A weighted space is $\sigma$-compact, i.e.~a
countable union of compacts. In the following remark we clarify the
question of local compactness of convex subsets
$\mathcal{E} \subset X$ when $X$ is a locally convex topological space
and $\rho$ convex.

\begin{remark}\label{rem:local_compactness}
  Let $X$ be a locally convex topological space and $\mathcal{E}$ a
  convex subset.  Moreover, let $\rho$ be a \emph{convex} admissible
  weight function.  Then $ \rho $ is continuous on $ \mathcal{E} $ if
  and only if $\mathcal{E}$ is locally compact. Indeed if $ \rho $ is
  continuous on $\mathcal{E}$, then of course the topology on
  $ \mathcal{E} $ is locally compact since every point has a compact
  neighborhood of type $ \{ \rho \leq R \} $ for some $ R > 0 $. On
  the other hand if the topology on $ \mathcal{E} $ is locally
  compact, then for every point $ \lambda_0 \in \mathcal{E} $ there is
  a a convex, compact neighborhood $ V \subset \mathcal{E}$ such that
  $ \rho(\lambda)-\rho(\lambda_0) $ is bounded on $ V $ by a number
  $ k > 0 $, whence by convexity
  $ |\rho(s(\lambda-\lambda_0)+\lambda_0)-\rho(\lambda_0)| \leq s k $
  for $ \lambda - \lambda_0 \in s(V-\lambda_0) $ and $ s \in ]0,1]
  $. This in turn means that $ \rho $ is continuous at $ \lambda_0 $.
\end{remark}

From now on $\rho$ shall always denote an admissible weight
function. For completeness we start by putting definitions for general
Banach space valued functions, although in the sequel we shall only
deal with $\mathbb{R}$-valued functions: let $Z$ be a Banach space
with norm ${\lVert \cdot \rVert}_Z $. The vector space
\begin{equation}
  \mathrm{B}^\rho(X;Z):=\left\{ f\colon X\to Z \colon \sup_{x\in X}\rho(x)^{-1}{\lVert f(x)\rVert}_Z < \infty \right\}
\end{equation}
of $ Z $-valued functions $ f $ equipped with the norm
\begin{equation}
  \lVert f\rVert_{\rho}:=\sup_{x\in X}\rho(x)^{-1}{\lVert f(x) \rVert}_Z,
\end{equation}
is a Banach space itself. It is also clear that for $Z$-valued bounded
continuous functions the continuous embedding
$\mathrm{C}_b(X;Z)\subset\mathrm{B}^\rho(X;Z)$ holds true, where we
consider the supremum norm on bounded continuous functions,
i.e. $\sup_{x \in X}\|f(x)\|$.
\begin{definition}
  We define $\mathcal{B}^{\rho}(X;Z)$ as the closure of
  $\mathrm{C}_b(X;Z)$ in $\mathrm{B}^{\rho}(X;Z)$. The normed space
  $\mathcal{B}^{\rho}(X;Z)$ is a Banach space.
\end{definition}

If the range space $Z=\mathbb{R}$, which from now on will be the case,
we shall write $ \mathcal{B}^\rho(X) $ for
$\mathcal{B}^{\rho}(X; \mathbb{R})$ and analogously $B^{\rho}(X)$.

We consider elements of $ \mathcal{B}^\rho(X) $ as continuous
functions whose growth is controlled by $ \rho $. More precisely we
have by \cite[Theorem 2.7]{doetei:10} that $f\in\mathcal{B}^{\rho}(X)$
if and only if $f|_{K_R}\in \mathrm{C}(K_R)$ for all $R>0$ and
\begin{equation}
  \label{eq:Bdecay}
  \lim_{R\to\infty}\sup_{x\in X\setminus K_R}\rho(x)^{-1}\lVert f(x)\rVert=0 \, .
\end{equation}
Additionally, by \cite[Theorem 2.8]{doetei:10} it holds that for every
$f\in\mathcal{B}^\rho(X)$ with\\ $\sup_{x\in X}f(x)>0$, there exists
$z\in X$ such that
\begin{equation}
  \rho(x)^{-1}f(x) \le \rho(z)^{-1}f(z) \quad\text{for all $x\in X$},
\end{equation}
which emphasizes the analogy with spaces of continuous functions
vanishing at $ \infty $ on locally compact spaces.

Let us now state the following crucial representation theorem of Riesz
type:
\begin{theorem}[Riesz representation for
  $\mathcal{B}^\rho(X)$]\label{theorem:rieszrepresentation}
  For every continuous linear functional
  $\ell\colon\mathcal{B}^\rho(X)\to\mathbb{R}$ there exists a finite
  signed Radon measure $\mu$ on $X$ such that
  \begin{equation}
    \ell(f)=\int_{X}f(y)\mu(dy)\quad\text{for all $f\in\mathcal{B}^\rho(X)$.}
  \end{equation}
  Additionally
  \begin{equation}
    \label{eq:rieszrepresentation-psiintbound}
    \int_{X}\rho(y)\lvert \mu\rvert( dy) = \lVert \ell\rVert_{L(\mathcal{B}^\rho(X),\mathbb{R})},
  \end{equation}
  where $\lvert\mu\rvert$ denotes the total variation measure of
  $\mu$.
\end{theorem}

We shall next consider positive, strongly continuous semigroups on
$ \mathcal{B}^\rho(X) $ and recover very similar structures as for
standard Feller semigroups on the space of continuous functions
vanishing at $ \infty $ on locally compact spaces.

\begin{definition}\label{def:genFeller}
  A family of bounded linear operators
  $P_t\colon\mathcal{B}^{\rho}(X)\to\mathcal{B}^{\rho}(X)$ for
  $ t \geq 0 $ is called \emph{generalized Feller semigroup} if
  \begin{enumerate}
    \renewcommand{\theenumi}{{\bf F\arabic{enumi}}}
  \item
    \label{enu:defgenfeller-0id}
    $P_0=I$, the identity on $\mathcal{B}^{\rho}(X)$,
  \item $P_{t+s}=P_tP_s$ for all $t$, $s\ge 0$,
  \item
    \label{enu:defgenfeller-pwconv}
    for all $f\in\mathcal{B}^{\rho}(X)$ and $x\in X$,
    $\lim_{t\to 0}P_t f(x)=f(x)$,
  \item
    \label{enu:defgenfeller-bound}
    there exist a constant $C\in\mathbb{R}$ and $\varepsilon>0$ such
    that for all $t\in [0,\varepsilon]$,
    $\lVert P_t\rVert_{L(\mathcal{B}^{\rho}(X))}\le C $.
  \item
    \label{enu:defgenfeller-positivity}
    $P_t$ is positive for all $t\ge 0$, that is, for
    $f\in\mathcal{B}^{\rho}(X)$, $f\ge 0$, we have $P_t f\ge 0$.
  \end{enumerate}
\end{definition}

We obtain due to the Riesz representation property the following key
theorem:
\begin{theorem}
  \label{theorem:Ttstrongcont}
  Let $(P_t)_{t\ge 0}$ satisfy (i) to (iv) of Definition
  \ref{def:genFeller}.  Then, $(P_t)_{t\ge 0}$ is strongly continuous
  on $\mathcal{B}^{\rho}(X)$, that is,
  \begin{equation}
    \lim_{t\to 0}\lVert P_t f-f\rVert_{\rho}=0
    \quad\text{for all $f\in\mathcal{B}^{\rho}(X)$}.
  \end{equation}
\end{theorem}
One can also establish a positive maximum principle in case that the
semigroup $ (P_t)_{t \geq 0} $ grows around $ 0 $ like
$ \exp(\omega t) $ for some $\omega \in \mathbb{R}$ with respect to
the operator norm on $ \mathcal{B}^{\rho}(X) $. Indeed, the following
theorem proved in \cite[Theorem 3.3]{doetei:10} is a reformulation of
the Lumer-Philips theorem for quasi-contractive semigroups using a
\emph{generalized positive maximum principle} which is formulated in
the sequel.
\begin{theorem}
  \label{theorem:Ttposmaxprinciple}
  Let $A$ be an operator on $\mathcal{B}^{\rho}(X)$ with domain $D$,
  and $\omega\in\mathbb{R}$. Then, $A$ is closable and its closure
  $\overline{A}$ generates a generalized Feller semigroup
  $(P_t)_{t\ge 0}$ with
  $\lVert P_t\rVert_{L(\mathcal{B}^{\rho}(X))}\le\exp(\omega t)$ for
  all $t\ge 0$ if and only if
  \begin{enumerate}
  \item $D$ is dense,
  \item $A-\omega_0$ has dense image for some $\omega_0>\omega$, and
  \item $A$ satisfies the generalized positive maximum principle, that
    is, for $f\in D$ with $(\rho^{-1}f)\vee 0\le \rho(z)^{-1}f(z)$ for
    some $z\in X$, $Af(z)\le \omega f(z)$.
  \end{enumerate}
\end{theorem}

\begin{remark}\label{rem:growthbound}
  Let $(P_t)_{t \geq 0}$ be a generalized Feller semigroup on
  $\mathcal{B}^{\rho}(X)$, then by abstract semigroup theory, see
  \cite[Proposition I.5.5]{engnag:00}, there exists $ M \geq 1 $ and
  $ \omega \in \mathbb{R} $ such that
  $ \| P_t \|_{L(\mathcal{B}^{\rho}(X))} \leq M \exp (\omega t ) $. By
  the Riesz representation as stated in Theorem
  \ref{theorem:rieszrepresentation}, $ P_t \rho (x) $ is always
  defined as integral of the measurable function $ \rho $ with respect
  to the measure representing the linear functional
  $ f \mapsto P_t f(x) $. The value $ P_t \rho (x) $ corresponds
  precisely to the operator norm of this linear functional by Equation
  \eqref{eq:rieszrepresentation-psiintbound} (note that the measure in
  Theorem \ref{theorem:rieszrepresentation} depends additionally on
  $x$). Hence by
  \[
    | P_t f (x) | \leq M \exp (\omega t ) \rho(x) {\lVert f
      \rVert}_\rho
  \]
  it follows that $ P_t \rho \leq M \exp(\omega t) \rho $.
\end{remark}
\begin{remark}\label{rem:quasicontractive}
  Let $ (P_t)_{t \geq 0} $ be a generalized Feller semigroup on
  $\mathcal{B}^{\rho}(X)$ of transport type, i.e.
  \begin{align}\label{eq:transport}
    P_t f (x) = f ( \psi_t(x))
  \end{align}
  for some continuous map $ \psi_t : X \to X $. Define now a new
  function
  \[
    \tilde \rho (x) := \sup_{t \geq 0} \, \exp(-\omega t) P_t \rho(x)
  \]
  for $ x \in X $. Notice that $ \tilde \rho $ is an admissible weight
  function, since
  \[
    \{ \tilde \rho \leq R \} = \cap_{t \geq 0} \, \{ P_t \rho \leq
    \exp(\omega t) R \} 
  \]
  is compact by the continuity of $x \mapsto \psi_t(x)$ as an
  intersection of closed subsets of compacts. Additionally we have
  that
  \[
    \rho \leq \tilde \rho \leq M \rho
  \]
  by the growth bound as in Remark \ref{rem:growthbound} and therefore
  the norm on $ \mathcal{B}^\rho(X) $ is equivalent to
  \[ {\lVert f \rVert}_{\tilde \rho} = \sup_{x \in X}
    \frac{|f(x)|}{\tilde \rho(x)} \, .
  \]
  Furthermore,
  \[
    \lVert P_tf \rVert_{\tilde \rho} \leq \exp(\omega t) \lVert f
    \rVert_{\tilde \rho}
  \]
  holds for all $t\ge 0$ and $ f \in \mathcal{B}^\rho(X) $.  Indeed,
  this is a consequence of the following estimate
  \begin{equation}
    \begin{split}\label{eq:estimate}
      \lVert P_tf \rVert_{\tilde \rho} &=
      \sup_{x}\left|\frac{f(\psi_t(x))}{\sup_s \exp(-\omega s)
          \rho(\psi_s(x))}\right|\leq
      \sup_{x}\left|\frac{f(\psi_t(x))}{\sup_s \exp(-\omega (t+s))
          \rho(\psi_{t+s}(x))}\right|
      \\
      &\leq \exp(\omega t) \sup_{x}\left|\frac{f(\psi_t(x))}{\sup_s
          \exp(-\omega s) \rho(\psi_{s}(\psi_t(x)))}\right|\leq
      \exp(\omega t) \|f\|_{\tilde{\rho}}.
    \end{split}
  \end{equation}
  Hence,
  \[
    |P_tf(x)| \leq \exp(\omega t)\tilde{\rho}(x) \|f\|_{\tilde{\rho}},
  \]
  which implies
  \[
    P_t \tilde \rho \leq \exp(\omega t) \tilde \rho, \quad t \geq 0.
  \]
\end{remark}

The next theorem is a basic perturbation result in the spirit of
\cite[Corollary 7.2]{ethkur:86} for generalized Feller semigroups.

\begin{theorem}\label{th:perturbation}
  Let $ A $ be a generator of a generalized Feller semigroup
  $(P_t)_{t\ge 0}$ on $ \mathcal{B}^\rho(X) $ with
  $\lVert P_t\rVert_{L(\mathcal{B}^{\rho}(X))}\leq M\exp(\omega t)$
  for all $t\ge 0$ and some $ M \geq 1$, $ \omega \in \mathbb{R}
  $. Let $ B $ be a bounded generator of a generalized Feller
  semigroup on $ \mathcal{B}^\rho(X) $. Then $ A + B $ is a generator
  of a generalized Feller semigroup on $ \mathcal{B}^\rho(X) $ whose
  operator norm is bounded by
  $ M \exp(\omega t + M \| B \|_{L(\mathcal{B}^{\rho}(X))} t) $ for
  $ t \geq 0 $. If $ M = 1 $ and $B$ satisfies
  \[
    B \rho \leq \tilde \omega \rho \,
  \]
  for some real $ \tilde \omega $, then the operator norm of the
  semigroup generated by $A+B$ is bounded by
  $ \exp((\omega + \tilde \omega) t) $ for $ t \geq 0 $.
\end{theorem}
\begin{proof}
  The result is a consequence of \cite[Theorem 1.3, p.158]{engnag:00}.
  The positivity is obtained from Chernoff's product formula. If
  $ B \rho \leq \tilde \omega \rho $, then for $ f \geq 0 $
  \[
    \exp( B t ) f \leq \exp( \tilde \omega t) {\|f \|}_{\rho} \rho \,
    ,
  \]
  which yields
  $\| B \|_{L(\mathcal{B}^{\rho}(X))} \leq \tilde{\omega}$.
\end{proof}

We now give a version of the Kolmogorov extension theorem, which has
not been shown so far. Notice that standard versions of the extension
theorem, i.e.~for Borel spaces or in case of inner-regular measures on
topological spaces, see \cite{alibor:06}, do not apply in this
case. One can, however, apply the Daniell-Kolmogorov theorem.

\begin{theorem}\label{th:kolmogorov_extension}
  Let $ (P_t)_{t \geq 0} $ be a generalized Feller semigroup with
  $ P_t 1 = 1 $ for $ t \geq 0 $. Then there exists a filtered
  measurable space $ (\Omega,(\mathcal{F}_t)_{t \geq 0}) $ with right
  continuous filtration, and an adapted family of random variables
  $ {(\lambda_t)}_{t \geq 0} $ such that for any initial value
  $ \lambda_0 \in X $ there exists a probability measure
  $ \mathbb{P}^{\lambda_0} $ with
  \[
    \mathbb{E}_{\mathbb{P}^{\lambda_0}}[f(\lambda_t)] = P_t
    f(\lambda_0)
  \]
  for $ t \geq 0 $ and every $ f \in \mathcal{B}^\rho(X) $. The Markov
  property holds true, i.e.
  \[
    \mathbb{E}_{\mathbb{P}^{\lambda_0}}[f(\lambda_t) \, | \;
    \mathcal{F}_s] = P_{t-s} f(\lambda_s)
  \]
  almost surely with respect to $ \mathbb{P}^{\lambda_0} $.
\end{theorem}
\begin{remark}
  For simplicity of notation we shall usually only write $\mathbb{E}$
  or $\mathbb{E}_{\lambda_0}$ for expectations with respect to the
  above probability measure $ \mathbb{P}^{\lambda_0}$.  We call the
  process $ (\lambda_t)_{t \geq 0} $ a generalized Feller process with
  initial value $ \lambda_0 $ with respect to this measure.
\end{remark}
\begin{proof}
  Every weighted space is $ \sigma $-compact and therefore the set of
  compact subsets of $ \{ \rho \leq R \} $ for $ R \geq 0 $ form a
  compact class in $ X $. Consider the kernels
  $ \mu_t(\lambda_0,\cdot) $ representing
  $ f \mapsto P_t f(\lambda_0) $ and fix $ \lambda_0 \in X $.  Then by
  standard constructions we obtain a consistent family of probability
  measures $ \mu^J $ on $ X^J $ for all finite subsets
  $ J \subset I $, i.e.~for the canonical projections
  $ \pi^{J_2 \to J_1}: X^{J_2} \to X^{J_1} $ the push forward of
  $ \mu^{J_2} $ is $ \mu^{J_1} $, for all finite sets
  $ J_1 \subset J_2 \subset I $. The Daniell-Kolmogorov theorem, see
  \cite[Theorem 15.23]{alibor:06}, then yields a probability measure
  on the $ X^{\mathbb{R}_{+}} $ on the $\sigma$ algebra generated by
  the canonical projections, such that the canonical projections
  $ \lambda_t : X \to \mathbb{R} $ have law $ \mu_t(\lambda_0,.)
  $. The rest follows by construction as in the Feller case. Right
  continuity of the filtration is enforced by taking the right
  continuous enlargement.
\end{proof}

We end this subsection with a statement on path properties: up to
explosion generalized Feller processes have -- with respect to
countably many test functions -- c\`adl\`ag or c\`agl\`ad
trajectories.

\begin{theorem}\label{th:path_properties}
  Let $ (P_t)_{t\geq 0} $ be a generalized Feller semigroup and let
  $ (\lambda_t)_{t \geq 0} $ be a generalized Feller process on a
  filtered probability space. Then for every countable family
  $ {(f_n)}_{n \geq 0} $ of functions in $ \mathcal{B}^\rho(X) $ we
  can choose a version of the processes
  $ {\left( \frac{f_n(\lambda_t)}{\rho(\lambda_t)} \right)}_{t \geq 0}
  $, such that the trajectories are c\`agl\`ad for all $ n \geq 0 $.

  If additionally $ P_t \rho \leq \exp(\omega t) \rho $ holds true,
  then $ (\exp(- \omega t) \rho(\lambda_t))_{t \geq 0} $ is a
  super-martingale and can be chosen to have c\`agl\`ad
  trajectories. In this case we obtain that the processes
  $ {\big( f_n(\lambda_t) \big)}_{t \geq 0} $ can be chosen to have
  c\`agl\`ad trajectories.
\end{theorem}

\begin{remark}
  \begin{enumerate}
  \item In the general case, when
    $ P_t \rho \leq M \exp(\omega t) \rho $ for $M >1$, we obtain for
    $ {\big( f_n(\lambda_t) \big)}_{t \geq 0} $ only c\`ag
    trajectories. To see this, consider the measurable set of sample
    events $ \{ \sup_{0 \leq t \leq 1} \rho(\lambda_t) \leq R \} $. On
    this set we can construct a c\`agl\`ad version of the processes
    $ {\left( \frac{f_n(\lambda_t)}{\rho(\lambda_t)} \right)}_{t \leq
      1} $ and $ \left({\frac{1}{\rho(\lambda_t)}}\right)_{t \leq 1} $
    and in turn also of $ {\big( f_n(\lambda_t) \big)}_{t \geq
      0}$. The limit $ R \to \infty $, however, only leads to a c\`ag
    version since we cannot control the right limits.
  
  \item Notice that we do not speak -- in this general context --
    about path properties of $ {(\lambda_t)}_{t \geq 0} $ itself. In
    case of separability of $ \mathcal{B}^\rho(X) $ we can actually
    find a c\`ag version of $ \lambda $ itself by choosing as point
    separating family $ f_n $ a countable dense subset of
    $ \mathcal{B}^\rho(X) $.
  \end{enumerate}

\end{remark}

\begin{proof}
  For every generalized Feller semigroup the process
  \[
    Y^\alpha_t := \exp(-\alpha t) R(\alpha)f(\lambda_t)
  \]
  for $ t \geq 0 $ is actually a non-negative supermartingale for
  $ f \geq 0 $ in $ \mathcal{B}^\rho(X) $. Here
  \[
    R(\alpha):= \int_0^\infty \exp(-\alpha s) P_s ds
  \]
  denotes the resolvent of $ (P_t)_{t \geq 0} $, defined for large
  enough $ \alpha $, say $\alpha > \omega$ for some $\omega \geq
  0$. Indeed by the Markov property
  \begin{align*}
    \mathbb{E}[Y^\alpha_{t} | \mathcal{F}_s] & = \mathbb{E} \big[ \exp(-\alpha t) \int_0^\infty \exp(-\alpha u) P_u f (\lambda_t)du | \mathcal{F}_s \big] \\
                                             & = \exp(-\alpha t) \int_0^\infty \exp(-\alpha u) P_{u+t-s} f (\lambda_s)du \\
                                             & = \exp(-\alpha s) \int_{t-s}^\infty \exp(-\alpha u) P_{u} f (\lambda_s)du \\
                                             & \leq Y^\alpha_{s} \, ,
  \end{align*}
  for $ 0 \leq s \leq t$ with respect to the given right continuous
  filtration. Hence $(Y^{\alpha}_t)_{t \geq 0}$ is a supermartingale
  and $t \mapsto \mathbb{E}[Y^\alpha_t] $ is continuous. This implies
  together with the right continuity of the filtration the existence
  of a c\`agl\`ad version.  Consider now the countable set of
  functions
  \[
    \mathcal{H} := \{ \alpha R(\alpha) f_n | \text{ for } n \in
    \mathbb{N}, \, \alpha > \omega \} \, .
  \]
  Since each $(Y^{\alpha}_t)_{t \geq0}$ has a c\`agl\`ad version, this
  translates also to each of the processes
  $ {(h(\lambda_t))}_{t \geq 0} $ for $ h \in \mathcal{H}$.  Moreover,
  as $ {\| \alpha R(\alpha) f - f \|}_\rho \to 0 $ for
  $ \alpha \to \infty $, also
  $ {\big( \frac{f_n(\lambda_t)}{\rho(\lambda_t)}\big) }_{t \geq 0} $
  has c\`agl\`ad trajectories. The second statement follows directly.
\end{proof}

The following example shows that the c\`agl\`ad property can fail when
we have jump process with a singular kernel even though the general
Feller property holds true.

\begin{example}
  Let $X =\mathcal{M}(\mathbb{R}_{+} \cup \{+\infty\})$ be equipped
  with the total variation norm as normed space, or with the
  weak-$*$-topology. Take a standard Poisson process $ N $ with
  intensity $1$ and a measure $ \nu $ such that
  \[
    K(t) := \int_0^\infty \exp(-x t) \nu(dx) < \infty
  \]
  for $ t > 0 $, but with possibly infinite total mass. Consider now
  the (formal) SPDE
  \[
    d \lambda_t(dx) = - x \lambda_t(dx) dt + \nu dN_t \, ,
  \]
  whose explicit solution is given by variation of constants
  \[
    \lambda_t(dx) = \exp(-xt) \lambda_0(dx) + \int_0^t \exp(-x(t-s))
    \nu(dx) \, dN_s
  \]
  understood in the weak-$*$-topology, for $ t \geq 0 $. Take
  $ \rho(\lambda) = 1 + \int_0^\infty \lambda (dx) $ the shifted total
  variation norm, which makes $X$ a weighted space. Then
  \[
    P_t f(\lambda_0) = \mathbb{E}\big [ f(\lambda_t) \big]
  \]
  actually defines a generalized Feller semigroup if, e.g.,
  $ \int_0^t K(s) ds \leq M \exp(\omega t) $. Indeed, we have
  \[
    P_t \rho(\lambda_0) = 1 +\int_0^\infty \exp(-xt) \lambda_0(dx) +
    \int_0^t K(t-s) ds \leq M \exp( \omega t ) \rho(\lambda_0)
  \]
  for some (possibly new) constant $ M \geq 1 $ and some real
  $ \omega $. This together with point-wise continuity on cylindrical
  Fourier basis elements yields the generalized Feller property. The
  process
  \begin{align*}
    \rho(\lambda_t) &= 1+ \int_0^\infty \exp(-xt) \lambda_0(dx) +
                      \int_0^t K(t-s) dN_s \\
                    &=1+ \int_0^\infty \exp(-xt) \lambda_0(dx) + \sum_{{\tau_n} <t} K(t-\tau_n),
  \end{align*}
  where $\tau_n$ denotes the jump times of the Poisson process,
  however does only allow for a c\`ag version, if $K$ is singular.
\end{example}

\begin{remark}
  Theorem \ref{th:path_properties} allows to formulate a martingale
  problem in this general context if
  $P_t\rho \leq \exp(\omega t) \rho$. Indeed, let $A$ be the generator
  of a generalized Feller semigroup. Then, for every
  $ f \in \operatorname{dom}(A) $ we can choose versions for
  $ (f(\lambda_t))_{t \geq 0} $ and $ (Af(\lambda_t))_{t \geq 0} $
  which are c\`agl\`ad. Hence $ \int_0^t A f(\lambda_s) ds$ is in
  particular well-defined and the Markov property implies that
  \[ \left( f(\lambda_t) - \int_0^t A f(\lambda_s) ds \right)_{t \geq
      0}
  \]
  is actually a (c\`agl\`ad) martingale. This will be investigated
  further in a subsequent article.
\end{remark}

\subsection{Dual spaces of Banach spaces} \label{subsec:dual}

The most important playground for our theory will be closed subsets of
duals of Banach spaces, where the weak-$*$-topology appears to be
$ \sigma $-compact due to the Banach-Alaoglu theorem. Assume that
$ \mathcal{E} \subset Y^*$ is a closed subset of the dual space $Y^*$
of some Banach space $Y$ where $Y^{\ast}$ is equipped with its
weak-$*$-topology. Consider a lower semicontinuous function
$\rho\colon \mathcal{E} \to(0,\infty)$ and denote by
$(\mathcal{E},\rho)$ the corresponding weighted space. We have the
following approximation result (see \cite[Theorem 4.2]{doetei:10}) for
functions in $\mathcal{B}^{\rho}(\mathcal{E})$ by cylindrical
functions. Set
\begin{alignat}{2}
  \mathcal{Z}_N := \bigl\{
  g(\langle\cdot,y_1\rangle,\dots,\langle\cdot,y_N\rangle)\colon
  &\text{$g\in\mathrm{C}_b^{\infty}(\mathbb{R}^N)$} \notag
  \\
  &\text{and $y_j\in Y$, $j=1,\dots,N$} \bigr\},
\end{alignat}
where $\langle \cdot, \cdot \rangle$ denotes the pairing between $Y^*$
and $Y$. We denote by
$\mathcal{Z}:=\bigcup_{N\in\mathbb{N}}\mathcal{Z}_N$ the set of
bounded smooth continuous cylinder functions on $\mathcal{E}$. We can
prove the following theorem beyond any separability assumptions on
$ Y $.
\begin{theorem}
  \label{theorem:boundedweakcontapprox}
  The closure of $\mathcal{Z}$ in $\mathrm{B}^\rho(\mathcal{E})$
  coincides with $\mathcal{B}^\rho(\mathcal{E})$, whose elements
  appear to precisely the functions
  $f\in\mathcal{B}^{\rho}(\mathcal{E})$ which satisfy
  \eqref{eq:Bdecay} and that $f|_{K_R}$ is weak-$*$-continuous for any
  $R>0$.
\end{theorem}
\begin{proof}
  Since $ \mathcal{Z}|_{\{\rho \leq R\}} $ is a point separating
  algebra we can apply the Stone Weierstrass theorem on the compact
  sets $ {\{\rho \leq R\}} $ to obtain density of the restrictions in
  $ C_b(\{\rho \leq R\}) $ for any $ R \geq 0 $. Then we can apply
  \cite[Theorem 2.7]{doetei:10}.
\end{proof}

\begin{remark} 
  Of course we can consider also subsets of general cylindrical
  functions to serve the same purpose (we just need a
  Stone-Weierstrass theorem to be applicable), i.e.~the subset should
  be point separating and an algebra. This will play an important role
  in the case of affine processes where we can consider the linear
  span of all products of Fourier basis elements
  $ \exp(\langle \cdot, y \rangle) $ for $ y \in Y $.
\end{remark}

In the following we will give a theorem telling when the semigroup of
a Markov process is actually a generalized Feller semigroup. For the
dense subset appearing in this theorem we take in practice the set of
cylindrical function $\mathcal{Z}$ introduced above. For its
formulation we need the following assumptions.

\begin{assumption}\label{ass:generic}
  Let $(\lambda_t)_{t\ge 0}$ denote a time homogeneous Markov process
  on some stochastic basis
  $(\Omega,\mathcal{F}, (\mathcal{F}_t)_{t\ge 0},
  \mathbb{P}^{\lambda_0})$ with values in $\mathcal{E}$.
  
  Then we assume that
  \begin{enumerate}
  \item there are constants $C$ and $\varepsilon>0$ such that
    \begin{equation}
      \label{eq:markovpsibound}
      \mathbb{E}_{\lambda_0}[\rho(\lambda_t)]\le C\rho(\lambda_0)
      \quad\text{for all $\lambda_0\in \mathcal{E}$ and $t\in[0,\varepsilon]$};
    \end{equation}
  \item
    \begin{equation} \label{eq:contint} \lim_{t\to 0}
      \mathbb{E}_{\lambda_0}[f(\lambda_t))] = f(\lambda_0) \quad
      \text{for any $f\in\mathcal{B}^{\rho}(\mathcal{E})$ and
        $\lambda_0\in \mathcal{E}$};
    \end{equation}
  \item for all $f$ in a dense subset of
    $ \mathcal{B}^\rho(\mathcal{E}) $, the map
    $ \lambda_0 \mapsto \mathbb{E}_{\lambda_0}[f(\lambda_t)] $ lies in
    $ \mathcal{B}^\rho(\mathcal{E}) $.
  \end{enumerate}
\end{assumption}

\begin{remark}
  Of course inequality \eqref{eq:markovpsibound} implies that
  $ \lvert\mathbb{E}_{\lambda_0}[f(\lambda_t)]\rvert \leq C
  \rho(\lambda_0) $ for all $ f \in \mathcal{B}^{\rho}(\mathcal{E}) $,
  $ \lambda_0 \in \mathcal{E} $ and $ t \in [0,\varepsilon]$.
\end{remark}

\begin{theorem}
  \label{theorem:strongcontprocess}
  Suppose Assumptions \ref{ass:generic} hold true. Then
  $P_t f(\lambda_0):=\mathbb{E}_{\lambda_0}[f(\lambda_t)]$ satisfies
  the generalized Feller property and is therefore a strongly
  continuous semigroup on $\mathcal{B}^\rho(\mathcal{E})$.
\end{theorem}

\begin{proof}
  This follows from the arguments of \cite[Section 5]{doetei:10}.
\end{proof}

\section{Approximation theorems}\label{sec:approximationTheorems}

In order to establish existence of Markovian solutions for general
generators $A$ we could either directly apply Theorem
\ref{theorem:Ttposmaxprinciple}, where we have to assume that the
generator $A$ satisfies on a dense domain $D$ a generalized positive
maximum principle and that for at least one $ \omega_0 > \omega $ the
range of $ A - \omega_0 $ is dense, or we approximate a general
generator $A$ by (pure jump) generators $A^n $ and apply the following
(well known) approximation theorems:

\begin{theorem}\label{thm:approximation}
  Let $ (P_t^n)_{n \in \mathbb{N}, t\geq 0} $ be a sequence of
  strongly continuous semigroups on a Banach space $Z$ with generators
  $ (A^n)_{n \in \mathbb{N}} $ such that there are uniform (in $n$)
  growth bounds $ M \geq 1 $ and $ \omega \in \mathbb{R} $ with
  \begin{align}\label{eq:growthbounduni}
    \| P^n_t \|_{L(Z)} \leq M \exp(\omega t)
  \end{align}
  for $ t \geq 0 $. Let furthermore
  $ D \subset \cap_n \operatorname{dom}(A^n)$ be a dense subspace with
  the following three properties:
  \begin{enumerate}
  \item $D$ is an invariant subspace for all $ P^n $, i.e.~for all
    $ f \in D $ we have $ P^n_t f \in D $, for $ n \geq 0 $ and
    $ t \geq 0 $.
  \item There is a norm $ {\|.\|}_D $ on $ D $ such that there are
    uniform growth bounds with respect to $ {\|.\|}_D $, i.e.~there
    are $ M_D \geq 1 $ and $ \omega_D \in \mathbb{R} $ with
    \[ {\| P^n_t f \|}_D \leq M_D \exp(\omega_Dt) {\|f\|}_D
    \]
    for $ t \geq 0 $ and for $ n \geq 0 $.
  \item The sequence $ A^n f $ converges as $ n \to \infty $ for each
    $ f \in D $, in the following sense: there exists a sequence of
    numbers $ a_{nm} \to 0 $ as $ n,m \to \infty $ such that
    \[
      \| A^n f - A^m f \| \leq a_{nm} {\| f \|}_D
    \]
    holds true for every $ f \in D $ and for all $n,m$.
  \end{enumerate}
  Then there exists a strongly continuous semigroup
  $ (P_t^\infty)_{t \geq 0} $ with the same growth bound on $ Z $ such
  that $ \lim_{n \to \infty} P^n_t f = P^\infty_t f $ for all
  $ f \in Z $ uniformly on compacts in time and on bounded sets in
  $D$. Furthermore on $ D $ the convergence is of order $ O(a_{nm})
  $. If in addition for each $n \in \mathbb{N}$, $(P_t^n)_{t \geq 0}$
  is a generalized Feller semigroup, then this property transfers also
  to the limiting semigroup.
\end{theorem}
\begin{proof}
  Since $ D $ is invariant under $ P^n $, $ P^m$ and
  $ D \subset \cap_n \operatorname{dom}(A^n)$ we can take first
  derivatives and obtain
  \[
    \frac{d}{ds} P_{s}^n P^m_{t-s} f = P_{s}^n (A^n - A^m) P^m_{t-s} f
    \, ,
  \]
  where we use that $ A^n $ commutes with $ P^n $ on
  $ \operatorname{dom}(A^n) $, which yields for $ m,n \in \mathbb{N} $
  and $ t \geq 0 $
  \[
    P^n_t f - P^m_t f = \int_0^t P^n_{s}(A^n - A^m)P^m_{t-s}f ds \, .
  \]
  By \eqref{eq:growthbounduni}, (i), (iii) and (ii), we can estimate
  \begin{align*}
    \|P^n_t f - P^m_t f\| &\leq \int_0^t \|P^n_{s}(A^n - A^m)P^m_{t-s}f\| ds \\
                          &\leq  \int_0^t M \exp(\omega s) \|(A^n - A^m)P^m_{t-s}f\| ds\\
                          &\leq  \int_0^t M \exp(\omega s) a_{nm} {\| P^m_{t-s} f \|}_D ds\\
                          & \leq  a_{nm} M M_D \|f\|_D\int_0^t  \exp(\omega s +\omega_D(t-s))  ds.
  \end{align*}
  This yields uniform convergence of $ P^n_t f $ on compact intervals
  in time (with respect to the norm of $ Z $) for all $ f \in D
  $. Hence we obtain bounded linear operators $ P^\infty_t $
  satisfying the semigroup property and the same growth bound on $ Z $
  with constants $ \omega $ and $ M $, in particular
  $ \lim_{n \to \infty} P^n_t f = P^\infty_t f $ for all $ f \in Z
  $. The convergence rate on $D$ is obvious. Strong continuity follows
  by uniform convergence on $D$ on bounded sets and by the respective
  stability estimates. For generalized Feller semigroups, the only
  additional property is positivity, which of course remains in the
  limit.
\end{proof}

For the purposes of affine processes a slightly more general version
of the approximation theorem is needed, which we state in the sequel:

\begin{theorem}\label{thm:approximation_gen}
  Let $ (P_t^n)_{n \in \mathbb{N}, t\geq 0} $ be a sequence of
  strongly continuous semigroups on a Banach space $Z$ with generators
  $ (A^n)_{n \in \mathbb{N}} $ such that there are uniform (in $n$)
  growth bounds $ M \geq 1 $ and $ \omega \in \mathbb{R} $ with
  \[
    \| P^n_t \|_{L(Z)} \leq M \exp(\omega t)
  \]
  for $ t \geq 0 $. Let furthermore
  $ D \subset \cap_n \operatorname{dom}(A^n)$ be a \emph{subset} with
  the following two properties:
  \begin{enumerate}
  \item The linear span $\operatorname{span}(D)$ is dense.
  \item There is a norm $ {\|.\|}_D $ on $ \operatorname{span}(D) $
    such that for each $ f \in D $ and for $ t > 0 $ there exists a
    sequence $ a^{f,t}_{nm} $, possibly depending on $ f $ and $t >0$,
    \[
      \| A^n P^m_u f - A^m P^m_u f \| \leq a^{f,t}_{nm} {\| f \|}_D
    \]
    holds true for $ n, m $ and for $ 0 \leq u \leq t$, with
    $ a^{f,t}_{nm} \to 0 $ as $ n,m \to \infty $.
  \end{enumerate}
  Then there exists a strongly continuous semigroup
  $ (P_t^\infty)_{t \geq 0} $ with the same growth bound on $ Z $ such
  that $ \lim_{n \to \infty} P^n_t f = P^\infty_t f $ for all
  $ f \in Z $ uniformly on compacts in time. If in addition for each
  $n \in \mathbb{N}$, $(P_t^n)_{t \geq 0}$ is a generalized Feller
  semigroup, then this property transfers also to the limiting
  semigroup.
\end{theorem}
\begin{proof}
  Again we apply the following well known formula for $ f \in D $,
  $ m,n \geq 0 $ and $ t \geq 0 $
  \[
    P^n_t f - P^m_t f = \int_0^t P^n_{s}(A^n - A^m)P^m_{t-s}f ds.
  \]
  By \eqref{eq:growthbounduni}, (i), (iii) and (ii), we can estimate
  \begin{align*}
    \|P^n_t f - P^m_t f\| &\leq \int_0^t \|P^n_{s}(A^n - A^m)P^m_{t-s}f\| ds \\
                          &\leq  \int_0^t M \exp(\omega s) \|(A^n - A^m)P^m_{t-s}f\| ds\\
                          &\leq  \int_0^t M \exp(\omega s) a^{f,t}_{nm} {\| f \|}_D ds\\
  \end{align*}
  This yields uniform convergence of $ P^n_t f $ on compact intervals
  in time (with respect to the norm of $ Z $) for all $ f \in D $ by
  assumption. Hence we obtain bounded linear operators $ P^\infty_t $
  satisfying the semigroup property and the same growth bound on $ Z $
  with constants $ \omega $ and $ M $, in particular
  $ \lim_{n \to \infty} P^n_t f = P^\infty_t f $ for all $ f \in Z $
  by stability of growth bounds. Strong continuity follows by the
  respective stability estimates, i.e.~
  \[
    \| P^\infty_t f - f \| \leq \| P^\infty_t f - P^\infty_t g \| + \|
    P^\infty_t g - P^n_t g \| + \| P^n_t g - g \| + \| g - f \|
  \]
  for $ g \in \operatorname{span}(D) $, where the first term is small
  due to $ \| P^\infty_t \| \leq M \exp(\omega t) $ for $ t \geq 0 $.
  For generalized Feller semigroups, the only additional property is
  positivity, which of course remains in the limit.
\end{proof}

Our first application of Theorem \ref{thm:approximation} is the next
proposition, where generalized Feller processes of pure jump type with
unbounded activity are constructed.

\begin{proposition}\label{prop:jump_perturbation}
  Let $ (X,\rho) $ be a weighted space with weight function
  $ \rho \geq 1 $. Consider an operator $A$ on $\mathcal{B}^{\rho}(X)$
  with dense domain $\operatorname{dom}(A)$ generating on
  $ \mathcal{B}^\rho(X) $ a generalized Feller semigroup
  $ (P_t)_{t\geq 0} $ of transport type as in \eqref{eq:transport},
  such that for all $t \geq 0$ we have
  $ \|P_t\|_{L(B^{\rho}(X))} \leq M_1 \exp(\omega t)$ for some $M_1$
  and $ \omega $ and such that
  $ \mathcal{B}^{\sqrt{\rho}}(X) \subset \mathcal{B}^\rho(X) $ is left
  invariant and $(P_t)_{t\geq 0} $ is strongly continuous thereon.

  Consider furthermore a family of finite nonnegative measures
  $ \mu(x,.)$ for $ x \in X $ on $ X $ such that the operator $B$ acts
  on $\mathcal{B}^{\rho}(X)$ by
  \[
    B f (x) : = \int (f(y) - f(x)) \mu(x,dy)
  \]
  for $ x \in X $ yielding continuous functions on
  $ \{\rho \leq R \} $ for $ R \geq 0 $, and such that the following
  properties hold true:
  \begin{itemize}
  \item For all $ x \in X $
    \begin{align}\label{eq:cond1} \int \rho(y)
      \mu(x,dy) \leq M \rho^2 (x),
    \end{align}
    as well as
    \begin{align} \label{eq:cond2} \int \sqrt{\rho(y)} \mu(x,dy) \leq
      M \rho (x),
    \end{align}
    and
    \begin{align}\label{eq:cond3}
      \int \mu(x,dy) \leq M \sqrt{\rho (x)},
    \end{align}
    hold true for some constant $M$.
  \item For some constant $ \widetilde{\omega} \in \mathbb{R} $
    \begin{align}\label{eq:cond4}
      \int \Big | \frac{\sup_{t \geq 0} \exp(-\omega t) P_t \rho(y) -\sup_{t \geq 0} \exp(- \omega t ) P_t \rho(x)}{\sup_{t \geq 0} \exp(-\omega t) P_t \rho(x)} \Big | \mu(x,d y)
      \leq \widetilde{\omega} ,
    \end{align}
    for all $ x \in X $. In particular
    $ y \mapsto \sup_{t \geq 0} \exp(-\omega t) P_t \rho(y) $ should
    be integrable with respect to $ \mu(x,.) $
  \end{itemize}
  Then $ A + B $ generates a generalized Feller semigroup
  $(P_t^{\infty})_{t \geq 0}$ on $ \mathcal{B}^\rho(X) $ satisfying
  $\|P^{\infty}_t\|_{L(\mathcal{B}^{\rho}(X))} \leq M_1 \exp((\omega +
  \tilde{\omega})t)$.
\end{proposition}
\begin{proof}
  We apply Theorem \ref{thm:approximation} with
  $D=\operatorname{dom}(A) \cap \mathcal{B}^{\sqrt{\rho}}(X)$.  We
  construct a sequence of pure jump, bounded activity generators on
  $ \mathcal{B}^\rho (X) $ \emph{and} on
  $ \mathcal{B}^{\sqrt{\rho}}(X) $ by
  \[
    B^n f (x) : = \int (f(y) - f(x)) \frac{n}{\rho(x) \lor n}
    \mu(x,dy) \, .
  \]
  Indeed, they are bounded with respect to both norms: by
  \eqref{eq:cond1} and \eqref{eq:cond3}, we have for
  $f \in \mathcal{B}^\rho (X) $
  \begin{align*}
    \frac{B^nf(x)}{\rho(x)} &\leq \frac{ \int ( \frac{f(y)}{\rho(y)}
                              \rho(y) -f(x) ) \frac{n}{\rho(x) \lor n}\mu(x,dy) }{\rho(x)}\\
                            & \leq \|f\|_{\rho} 2Mn 
  \end{align*}
  and by \eqref{eq:cond2} and \eqref{eq:cond3} for
  $f \in \mathcal{B}^{\sqrt{\rho}}(X) $
  \begin{align*}
    \frac{B^nf(x)}{\sqrt{\rho(x)}} &\leq \frac{ \int ( \frac{f(y)}{\sqrt{\rho(y)}}
                                     \sqrt{\rho(y)} -f(x) ) \frac{n}{\rho(x) \lor n}\mu(x,dy) }{\sqrt{\rho(x)}}\\
                                   & \leq \|f\|_{\sqrt{\rho}} Mn( \sqrt{n}+1).
  \end{align*}
  Consider now the operators $A^n:=A+B^n$, on the one hand with domain
  $ \operatorname{dom}(A) $ and on the other hand with domain
  $D= \operatorname{dom}(A) \cap \mathcal{B}^{\sqrt{\rho}}(X) $.  By
  Theorem \ref{th:perturbation} they generate generalized Feller
  semigroups $P^n$ on $\mathcal{B}^{\rho}(X)$ and
  $\mathcal{B}^{\sqrt{\rho}}(X)$ respectively by assumption.  Since
  the domain of the generator is anyhow left invariant, Condition (i)
  of Theorem \ref{thm:approximation} is satisfied with
  $D=\operatorname{dom}(A) \cap \mathcal{B}^{\sqrt{\rho}} (X) $.
  Moreover, by \eqref{eq:cond2} and \eqref{eq:cond3} we have for
  $f \in D$
  \begin{align*}
    \big | \frac{B^nf(x) - B^mf(x)}{\rho(x)} \big | & \leq {\| f
                                                      \|}_{\sqrt{\rho}}\frac{ \int \sqrt{\rho(y)} | \frac{n}{\rho(x)
                                                      \lor n} - \frac{m}{\rho(x) \lor m} |\mu(x,dy) }{\rho(x)} + \\
                                                    & + {\| f
                                                      \|}_{\sqrt{\rho}}\frac{ \int  | \frac{n}{\rho(x)
                                                      \lor n} - \frac{m}{\rho(x) \lor m} |\mu(x,dy) }{\sqrt{\rho(x)}}
                                                      \leq
                                                      a_{nm} {\| f \|}_{\sqrt{\rho}}
  \end{align*}
  for some $ a_{nm} \to 0 $ as $ n,m \to \infty $. Hence we also have
  that
  \[ {\| A f + B^n f - (A f + B^m f) \|}_\rho \leq a_{nm} {\| f
      \|}_{\sqrt{\rho}}
  \]
  for all $ f \in D $, implying that Condition (iii) of Theorem
  \ref{thm:approximation} is satisfied.

  We finally have to check whether the growth bounds on
  $ \mathcal{B}^\rho(X) $ and
  $ D \subseteq \mathcal{B}^{\sqrt{\rho}}(X) $ are uniform in $n$. In
  view of Remark \ref{rem:quasicontractive} we denote
  \[
    \tilde \rho(y) = \sup_{t \geq 0} \exp(-\omega t) P_t \rho(y)
  \]
  and apply Condition \eqref{eq:cond4} and the following immediate
  consequence of Condition \eqref{eq:cond4}
  \begin{align*}
    \int | \frac{\sqrt{\tilde \rho(y)} -\sqrt{\tilde \rho(x)}}{\sqrt{\tilde \rho(x)}} |
    \mu(x,d y) &\leq \int \big | \frac{\sqrt{\tilde \rho(y)} -\sqrt{\tilde \rho(x)}}{\sqrt{\tilde \rho(x)}} \big | \frac{\sqrt{\tilde \rho(y)} +\sqrt{\tilde \rho(x)}}{\sqrt{\tilde \rho(x)}}
                 \mu(x,d y) \\
               &=\int | \frac{\tilde \rho(y) - \tilde \rho(x)}{\tilde \rho(x)} |\mu(x,d y)  \leq \widetilde{\omega}
                 , \quad \text{ for all } x \in X \, .
  \end{align*}
  These conditions yield that
  $B^n\tilde \rho(x) \leq \widetilde{\omega} \tilde \rho(x)$ and
  $B^n\sqrt{\tilde \rho(x)} \leq \widetilde{\omega} \sqrt{\tilde
    \rho(x)}$ for all $n \in \mathbb{N}$. By
  Remark~\ref{rem:quasicontractive} the transport semigroup $ P $ is
  quasicontractive with respect to $ \tilde \rho $ and also
  $\sqrt{\tilde{\rho}} $. The latter follows by an analogous estimate
  as in \eqref{eq:estimate}.

  Hence we can readily apply the second part of
  Theorem~\ref{th:perturbation} with $\rho$ ($\sqrt{\rho}$
  respectively) replaced by $\tilde{\rho}$ ($\sqrt{\tilde{\rho}}$
  respectively) to get the following uniform growth bound for
  $(P^n)_n$
  \[
    \| P_t^n\|_{(L(\mathcal{B}^{\rho}(X)),\tilde{\rho})} \leq
    \exp((\omega+\widetilde{\omega})t), \quad \|
    P_t^n\|_{(L(\mathcal{B}^{\sqrt{\rho}}(X)),\sqrt{\tilde{\rho}} )}
    \leq \exp((\frac{\omega}{2}+\widetilde{\omega})t)
  \]
  on $ \mathcal{B}^\rho (X) $
  ($D \subseteq \mathcal{B}^{\sqrt{\rho}} (X)$ respectively) equipped
  with the norm induced by $\tilde{\rho}$ ($\sqrt{\tilde{\rho}}$
  respectively). The latter is indicated in the subscript of the above
  norms.  Translating this back to the norm corresponding to $\rho$
  yields
  \begin{align*}
    \|P^n_tf\|_{\rho}&= \sup_{x \in X} \frac{|P^n_t f(x)|}{\tilde{\rho}(x)} \frac{\tilde{\rho}(x)}{\rho(x)}\leq \| P^n_t f\|_{\tilde{\rho}} \sup_{x \in X} \frac{\tilde{\rho}(x)}{\rho(x)} \\
                     &\leq \exp((\omega +\widetilde{\omega})t)\|f\|_{\tilde{\rho}} M_1 \leq  \exp((\omega +\widetilde{\omega})t)\|f\|_{\rho} M_1 ,
  \end{align*}  
  where we use $\rho \leq \tilde{\rho} \leq M_1 \rho$ as asserted in
  Remark \ref{rem:quasicontractive}.  Similarly we have
  \[
    \|P^n_tf\|_{\sqrt{\rho}}\leq \exp((\frac{\omega}{2}
    +\widetilde{\omega})t)\|f\|_{\sqrt{\rho}} \sqrt{M_1}.
  \]
  We therefore also obtain uniform growth bounds for the orginal
  norms. Hence the approximation Theorem \ref{thm:approximation} can
  be applied and leads to a generalized Feller semigroup
  $P^{\infty}_t$ with generator $A+B$ whose growth bound satisfies
  $\|P^{\infty}_t\|_{L(\mathcal{B}^{\rho}(X))} \leq M_1\exp((\omega +
  \tilde{\omega})t)$ as asserted.
\end{proof}

In contrast to classical Feller theory,
Proposition\ref{prop:jump_perturbation} allows to construct processes
with unbounded jump intensities directly if $ \rho $ is unbounded on
$ X $. The general character of the proposition allows to build
general processes from simple ones by perturbation.  The following
corollary shows this in the case of affine processes on
$\mathbb{R}^d$.

\begin{corollary}\label{cor:affine}
  Let $ X \subset \mathbb{R}^d $ be closed subset and let
  $ \rho(x) : = 1 + {\|x\|}^2 $ be a weight function. Let furthermore
  $ A $ be the generator of a generalized Feller semigroup
  $ (P_t)_{t\geq 0} $ on $ \mathcal{B}^\rho(X) $, which satisfies
  $\|P_t\|_{L(B^{\rho}(X))}\leq M_1\exp(\omega t)$ for some constants
  $M _1\geq 1$ and $\omega \in \mathbb{R}$ and which is of transport
  type as in \eqref{eq:transport} with $x\mapsto \psi_t(x)$ being an
  affine function.  Assume furthermore that it leaves
  $ \mathcal{B}^{\sqrt{\rho}}(X) \subset \mathcal{B}^\rho(X) $
  invariant and generates a strongly continuous semigroup there as
  well.  Let
  \[
    \mu(x,d\xi) := \sum_{i=1}^d x_i \mu^i(d\xi)
  \]
  for some possibly signed measures $ \mu^i $ with bounded support
  such that $ \mu(x,d\xi) $ defines a family of finite nonnegative
  measures on $X$ and $x+\operatorname{supp}(\mu(x, \cdot)) \in X$ for
  all $x \in X$. Then
  \[
    f \mapsto \big ( x \mapsto A f(x) + \int (f(x+\xi) - f(x))
    \mu(x,d\xi) \big)
  \]
  for $ f \in \operatorname{dom}(A)\cap \mathcal{B}^{\sqrt{\rho}}(X) $
  generates a generalized Feller semigroup on $ \mathcal{B}^\rho(X) $.
\end{corollary}
\begin{proof} Substituting $y=x+\xi$, one easily verifies that
  Conditions \eqref{eq:cond1} - \eqref{eq:cond3} of Proposition
  \ref{prop:jump_perturbation} are satisfied since the supports of
  $ \mu^i $ are bounded. Concerning Condition \eqref{eq:cond4}, denote
  \[
    \tilde{\rho}(x) =\sup_{t\geq 0} \exp(-\omega t) P_t \rho(x) \, .
  \]
  In particular we know that $ \rho \leq \tilde \rho$ and it holds
  that $P_tf(x)=f(\psi_t(x))$ where $\psi$ is an affine
  function. Using this together with
  $ | \sup_t c(t) - \sup_t d(t)| \leq \sup_t |c(t) - d(t) | $ we
  obtain for some $\widetilde{\omega}$
  \begin{align*}
    & \int  \big | \frac{\tilde{\rho}(x+  \xi)-\tilde{\rho}(x)}{\tilde{\rho}(x)} \big |  \sum_{i=1}^d x_i \mu^i(d\xi)
    \\&\quad  \leq\int \big| \frac{\sup_{t \geq 0} \exp(-\omega t) |P_t\rho(x + \xi) -  P_t \rho( x)|}{\tilde\rho(x)} \big |  \sum_{i=1}^d x_i \mu^i(d\xi)\\      
    &\quad \leq\int \big| \frac{\sup_{t \geq 0} \exp(-\omega t) |\rho(\psi_t(x + \xi)) -  \rho(\psi_t (x))|}{\tilde\rho(x)} \big |   \sum_{i=1}^d x_i \mu^i(d\xi)\\   
    &= \int \big | \frac{ \sup_{t \geq 0} \exp(-\omega t) C(\|\xi\|+ \| x \| \|  \xi \|
      + \| \xi \|^2 )}{1 +\|x\|^2} \big |\sum_{i=1}^d x_i \mu^i(d\xi)   \leq \widetilde{\omega} \, ,
  \end{align*}
  where $C$ denotes some constant. All other requirements are met as
  well and we can conclude.
\end{proof}

\section{Lifting Stochastic Volterra
  processes}\label{sec:markovianlift_abstract}

In the subsequent sections our main goal is to treat the following
types of SPDEs
\begin{equation}
  \begin{split} \label{eq:SPDE_weak}
    d \lambda_t &= \mathcal{A}^* \lambda_t dt + \nu dX_t, \\
    \lambda_0 & \in \mathcal{E} \, ,
  \end{split}
\end{equation}
on spaces $\mathcal{E} \subset Y^*$ as introduced in Section
\ref{subsec:dual} where $\mathcal{A}^*$ is the generator of a strongly
continuous semigroup $\mathcal{S}^*$ on $Y^*$, $\nu \in Y^*$ (or in a
slightly larger space denoted by $Z^*$ in the sequel), $g \in Y$ and
$X$ a real valued It\^o-semimartingale whose differential
characteristics depend linearly on $ \langle g , \lambda_t \rangle $,
which will turn out to be the solution of the Volterra equation with
kernel \eqref{eq:kernelintro}.

\begin{remark}
  As indicated in the introduction, it will be easy to consider vector
  valued structures, with $X$ a semimartingale whose characteristics
  depend -- instead of the $\mathbb{R}$-valued pairing
  $ \langle g , \lambda_t \rangle$ -- linearly on a projection of
  $ \lambda_t $ onto some space, finite or infinite dimensional. For
  the sake of the first exposition we stay one dimensional here.
\end{remark}

In the following we summarize the main ingredients of our setting.
\begin{assumption}\label{ass:weak_existence}
  Throughout this section we shall work under the following
  conditions:
  \begin{enumerate}
  \item We consider Banach spaces $Z$ and $Y$ such that
    $ Z \subset Y $ and $ Z $ embeds continuously into $Y$, and their
    duals $ Y^* \subset Z^* $ with their respective weak-$*$-topology.
  \item We are given an admissible weight function
    $ \rho = 1 + \rho_0 $ in the sense of Section~\ref{subsec:dual} on
    $ Y^* $ such that
    \[
      \rho_0(\lambda) = {\|\lambda\|}_{Y^*}^2, \quad \lambda \in Y^*,
    \]
    where $\|\cdot\|_{Y^*}$ denotes the norm on $Y^*$.
  \item We are given a closed convex cone $ \mathcal{E} \subset Y^* $
    such that $ (\mathcal{E},\rho) $ is a weighted space in the sense
    of Section \ref{sec:genFeller}.  This will serve as statespaces of
    \eqref{eq:SPDE_weak}.
  \item We assume that a semigroup $ \mathcal{S}^* $ with generator
    $ \mathcal{A}^* $ acts in a strongly continuous way on $ Y^* $ and
    $ Z^* $, with respect to the respective norm topologies.
  \item We assume that $ \lambda \mapsto \mathcal{S}^*_t\lambda $ is
    weak-$*$-continuous on $ Y^* $ and on $ Z^* $ for every
    $ t \geq 0 $ (considering the weak-$*$-topology on both the domain
    and the image space).
  \item We suppose that the (pre-) adjoint operator of
    $ \mathcal{A}^* $, denoted by $\mathcal{A}$ and domain
    $ \operatorname{dom}(\mathcal{A}) \subset Z \subset Y $, generates
    a strongly continuous semigroup on $Z$ with respect to the
    respective norm topology (but \emph{not} necessarily on $ Y $).
  \end{enumerate}
\end{assumption}

\begin{remark}
  We could allow more general weight functions $ \rho $ but it is not
  necessary for our purposes here and the formulation of their
  abstract properties is cumbersome.
\end{remark}

\begin{remark}\label{rem:examplespaces}
  A prototypical example presented in Section \ref{sec:markovianlift}
  is given by
  $ Y = C_b(\mathbb{R}_{+} \cup \{ \infty \}, \mathbb{R}) $ with
  supremum norm and $ Z $ the space of functions $ g \in Y $ such that
  $(x \mapsto x g(x) ) \in Y$ together with the operator norm on it,
  i.e.~ $ \| g \| = \sqrt{\| g \|^2 + \sup_{x \geq 0} | xg(x) |^2} $
  for $ g \in Z $.  The semigroup is given by
  $\mathcal{S}_t f(x) = \exp(-xt)f(x) $ for $ t \geq 0 $, $ f \in Y $
  and $ x \geq 0 $. All above requirements on spaces and semigroups
  are then satisfied. Notice also that we can consider non-separable
  settings since this is not a requirement in our setup.
\end{remark}

To analyze \eqref{eq:SPDE_weak} and to construct $ \mathcal{E} $ we
first consider the following linear deterministic equation
\begin{align}\label{eq:lambda_abstract}
  d \lambda_t = \mathcal{A}^* \lambda_t dt - w \nu \langle g , \lambda_t \rangle dt 
\end{align}
for $ \lambda_0 \in Y^* $, some fixed $ g \in Y $, a real number
$ w > 0 $ and $ \nu \in Z^* $.

Under the subsequent assumptions on $ \mathcal{S}^* $ and
$ \nu \in Z^* $ we can guarantee that it can be solved on the space
$Y^*$ for all times in the mild sense with respect to the dual norm
$\|\cdot\|_{Y^*}$ by a standard Picard iteration method.

\begin{assumption}\label{ass:semigroup}
  We assume that
  \begin{enumerate}
  \item $ \mathcal{S}^*_t \nu \in Y^* $ for all $ t > 0 $ even though
    $ \nu $ does not necessarily lie in $ Y^* $ itself, but only in
    $ Z^* $;
  \item $ \int_0^t \| \mathcal{S}^*_s \nu \|^2_{Y^*} ds < \infty $ for
    all $ t > 0 $.
  \end{enumerate}
\end{assumption}
As in \eqref{eq:kernelintro}, we define
\begin{align} \label{eq:kernel} K(t) := \langle g, S_t^{*} \nu
  \rangle,
\end{align}
which will correspond to the kernel in the Volterra equation
\eqref{eq:V} and define
$R \in L^2_{\text{loc}}(\mathbb{R}_+, \mathbb{R})$ as the resolvent of
the second kind that satisfies
\begin{align}\label{eq:resolvent}
  K \ast R=R\ast K=K-R.
\end{align}
This resolvent always exists and is unique (see \cite[Theorem
3.1]{gri:90}).

\begin{proposition}\label{prop:existenceY*} 
  Under Assumption \ref{ass:semigroup}, there exists a unique mild
  solution of \eqref{eq:lambda_abstract} with values in
  $Y^*$. Additionally, the solution operator is a weak-$*$-continuous
  map $ \lambda_0 \mapsto \lambda_t $, for each $ t > 0 $, and the
  solution satisfies
  \[
    \rho(\lambda_t) \leq C \rho(\lambda_0), \quad \text{for all }
    \lambda_0 \in Y^* \text{ and } t \in [0, \varepsilon]
  \]
  for some positive constants $ C $ and $ \varepsilon $.
\end{proposition}

\begin{remark}
  The unique mild solution of Equation \eqref{eq:lambda_abstract}
  satisfies
  \[
    \lambda_t = \mathcal{S}^*_t \lambda_0 - w \int_0^t
    \mathcal{S}^*_{t-s} \nu \langle g, \lambda_s \rangle ds
  \]
  for all $ t \geq 0 $. Pairing with $ g \in Y $ yields a
  deterministic linear Volterra equation of the form
  \begin{equation}\label{eq:detVolt}
    \begin{split}
      \langle g, \lambda_t \rangle = \langle g, \mathcal{S}_t^{*}
      \lambda_0 \rangle - w \int_0^t \langle g, \mathcal{S}_{t-s}^{*}
      \nu \rangle
      \langle g, \lambda_s \rangle ds \\
      =\langle g, \mathcal{S}_t^{*} \lambda_0 \rangle - w \int_0^t
      K(t-s) \langle g, \lambda_s \rangle ds,
    \end{split}
  \end{equation}
  where we used \eqref{eq:kernel}.
\end{remark}

\begin{proof}
  We prove first the completely standard convergence of the Picard
  iteration scheme with respect to the dual norm on $ Y^* $. Define
  \begin{align*}
    \lambda^0_t&= \lambda_0,\\
    \lambda^{n+1}_t&= \mathcal{S}^*_t \lambda_0 - w \int_0^t  \mathcal{S}^*_{t-s}  \nu \langle g , \lambda^n_s \rangle ds, \quad n \geq 0.
  \end{align*}
  Then, by Assumption \ref{ass:semigroup} (i) each $\lambda^n_t$ lies
  $Y^*$. Consider now
  \begin{align*}
    \|\lambda^{n+1}_t - \lambda^{n}_t\|_{Y^*} &= w\|\int_0^t  \mathcal{S}^*_{t-s}  \nu \langle g , \lambda^n_s - \lambda^{n-1}_s\rangle ds \|_{Y^*}\\
                                              & \leq w \int_0^t  \|\mathcal{S}^*_{t-s}  \nu \langle g , \lambda^n_s -\lambda^{n-1}_s  \rangle \|_{Y^*}  ds\\
                                              &\leq w \int_0^t  \|\mathcal{S}^*_{t-s}  \nu\|_{Y^*}  \|g\|_Y \|\lambda^n_s -\lambda^{n-1}_s  \|_{Y^*}   ds.
  \end{align*}
  Assumption \ref{ass:semigroup} (ii) and an extended version of
  Gronwall's inequality see \cite[Lemma 15]{D:99} then yield
  convergence of $(\lambda^n_t)_{n \in \mathbb{N}}$ to some
  $\lambda_t$ with respect to the dual norm $\|\cdot\|_{Y^*}$
  uniformly in $t$ on compact intervals. For details on strongly
  continuous semigroups and mild solutions see \cite{paz:83}.

  Having established the existence of a mild solution of
  \eqref{eq:lambda_abstract} in $Y^*$, consider now the linear,
  deterministic Volterra equation \eqref{eq:detVolt}, which can be
  written as in \cite[Section 2]{gri:90}, i.e.
  \begin{align}\label{eq:Volterra_det}
    \langle g, \lambda_t \rangle = \langle g, \mathcal{S}_t^{*} \lambda_0 \rangle - \int_0^t R^w(t-s) \langle g,\mathcal{S}_s^{*} \lambda_0\rangle ds,
  \end{align}
  where $R^w$ denotes the resolvent of
  $ w K(t)= w \langle g, S_{t}^{*} \nu \rangle$ as introduced in
  \eqref{eq:resolvent}. Since by assumption $ \mathcal{S}^* $ is a
  weak-$*$-continuous solution operator, the map
  $ \lambda_0 \mapsto (t \mapsto \langle g , \mathcal{S}^*_t \lambda_0
  \rangle ) $ is weak-$*$-continuous as a map from $ Y^* $ to
  $ C(\mathbb{R}_{+},\mathbb{R}) $ (with the topology of uniform
  convergence on compacts on $C(\mathbb{R}_{+},\mathbb{R})$). From
  \eqref{eq:Volterra_det} we thus infer that
  $ \langle g, \lambda_t \rangle $ is weak-$*$-continuous for every
  $ t \geq 0 $, which clearly translates to the solution map of
  Equation \eqref{eq:lambda_abstract}.

  Finally we have to show that the stated inequality for
  $ \rho(\lambda_t) $ holds true on small time intervals
  $ [0,\varepsilon]$. Observe first that for $t \in [0,\varepsilon]$
  \[
    \|\mathcal{S}^*_t \lambda\|_{Y^*}^2 \leq C \|\lambda\|^2_{Y^*}
  \]
  for all $\lambda \in Y^*$ just by the assumption that
  $\mathcal{S}^*_t$ is strongly continuous, for some constant
  $ C \geq 1 $. Furthermore for $t \in [0, \varepsilon]$
  \begin{align*}
    \|\lambda_t\|_{Y^*}^2 &\leq 2( C \|\lambda_0\|^2_{Y^*} + t\int_0^t \| w \mathcal{S}^*_{t-s} \nu \langle g , \lambda_s \rangle \|_{Y^*}^2 ds )\\
                          &\leq 2( C \|\lambda_0\|^2_{Y^*} + \varepsilon \int_0^t w^2 \|\mathcal{S}^*_{t-s} \nu\|^2_{Y^*} \|g\|^2_Y \|\lambda_s \|_{Y^*}^2  ds ).
  \end{align*}
  Consider now the kernel
  $K'(t,s)=2\varepsilon w^2 \|\mathcal{S}^*_{t-s} \nu\|^2_{Y^*}
  \|g\|^2_Y1_{\{s \leq t\}}$. We follow now the arguments of the proof
  of \cite[Lemma 3.1]{AbiLarPul:17}.  Indeed $K'$ is a Volterra kernel
  as defined in \cite[Definition 9.2.1]{gri:90} and for any interval
  $[U,V] \subset \mathbb{R}_+$ we have by Young's convolution
  inequality
  \[
    ||| K'|||_{L^1([U,V])} \leq 2\varepsilon w^2 \|g\|^2_Y
    \int_0^{V-U}\|\mathcal{S}^*_{s} \nu\|^2_{Y^*} ds,
  \]
  where $||| \cdot|||_{L^1([U,V])}$ is defined in \cite[Definition
  9.2.2]{gri:90}. We can now literally take the proof of \cite[Lemma
  3.1]{AbiLarPul:17} to deduce that the generalized Gronwall Lemma
  (see \cite[Lemma 9.8.2]{gri:90}) can be applied. This yields for
  $t \in [0,\varepsilon]$
  \[
    \|\lambda_t\|_{Y^*}^2 \leq \|\lambda_0\|^2_{Y^*} 2C(1 - \int_0^t
    R'(s) ds) \leq \|\lambda_0\|^2_{Y^*} 2C(1 - \int_0^{\varepsilon}
    R'(s) ds),
  \]
  where $R'$ denotes the resolvent of $-K'$, which is
  nonpositive. This leads to the desired assertion due to the
  definition of $\rho$. From this inequality also uniqueness follows
  in a standard way.
\end{proof}

It is of crucial importance to understand that there is actually a
closed sub-cone $ \mathcal{E} \subset Y^* $ left invariant by the
solution of Equation \eqref{eq:lambda_abstract}. This cone will play
the role as announced in Assumption \ref{ass:weak_existence} (iii) and
can be described in terms of initial values $\lambda_0$ which give
rise to nonnegative solutions of \eqref{eq:Volterra_det}. Indeed, it
will be the intersection of the following cones. Define for fixed
$ w > 0 $
\begin{equation}
  \begin{split}\label{eq:statespaceEw}
    \mathcal{E}^w:=& \{ \lambda_0 \in Y^* |  \langle g, \mathcal{S}_t^{*} \lambda_0 \rangle -  \int_0^t R^w(t-s) \langle g, \mathcal{S}_s^{*} \lambda_0\rangle ds \geq 0 \text{ for all }  t \geq 0  \}   \\
    = & \{ \lambda_0 \in Y^* | \langle g, \lambda_t \rangle \geq 0
    \text{ with } \lambda_t = \mathcal{S}^*_t \lambda_0 - w \int_0^t
    \mathcal{S}^*_{t-s} \nu \langle g, \lambda_s \rangle ds \text{ for
      all } t \geq 0 \},
  \end{split}
\end{equation}
where $R^w$ denotes the resolvent of
$ w K(t)= w \langle g, S_{t}^{*} \nu \rangle$.

\begin{proposition}\label{prop:invariance}
  Let Assumption \ref{ass:semigroup} be in force and let $ w > 0 $ be
  fixed. The set $\mathcal{E}^w$ as defined in \eqref{eq:statespaceEw}
  is a weak-$*$-closed convex cone.  The solution of
  \eqref{eq:lambda_abstract} leaves $ \mathcal{E}^w $ invariant and
  defines a generalized Feller semigroup on
  $\mathcal{B}^\rho(\mathcal{E}^w)$ by
  $ P_t f(\lambda_0) := f(\lambda_t) $ for all
  $ f \in \mathcal{B}^\rho(\mathcal{E}^w) $ and $ t \geq 0 $.
\end{proposition}

\begin{proof}
  The weak-$*$-closedness follows from the weak-$*$-continuity of
  $\mathcal{S}^*$ and the convex cone property is obvious. Let us now
  prove the invariance of $ \mathcal{E}^w $. Let
  $ \lambda_0 \in \mathcal{E}^w $.  We now have to prove that the for
  each $t > 0$ the unique solution of \eqref{eq:lambda_abstract} given
  by
  \[
    \lambda_t = \mathcal{S}^*_t \lambda_0 - w \int_0^t
    \mathcal{S}^*_{t-s} \nu \langle g, \lambda_s \rangle ds
  \]
  lies in $\mathcal{E}^w$ as well. By Definiton of $\mathcal{E}^w$,
  this means to show that $\langle g, \lambda_{t+u} \rangle \geq 0$
  for all $u \geq 0$ where
  \[
    \lambda_{t+u} = \mathcal{S}^*_u \lambda_t - w \int_0^u
    \mathcal{S}^*_{u-s} \nu \langle g, \lambda_{t+s} \rangle ds.
  \]
  This is equal to
  \[
    \lambda_{t+u}=\mathcal{S}^*_{t+u} \lambda_0 - w \int_0^{t+u}
    \mathcal{S}^*_{t+u-s} \nu \langle g, \lambda_{s} \rangle ds.
  \]
  Since $\lambda_0 \in \mathcal{E}^w$, we thus have
  $\langle g, \lambda_{t+u} \rangle \geq 0$ for all $u \geq 0$ and
  thus $\lambda_t \in \mathcal{E}^w$, proving the invariance.  Since
  by Proposition \ref{prop:existenceY*} the solution operator is
  weak-$*$-continuous, we can conclude that
  $\lambda_0 \mapsto f(\lambda_t)$ lies in
  $ \mathcal{B}^\rho(\mathcal{E}^w) $ for a dense set of
  $ \mathcal{B}^\rho(\mathcal{E}^w) $ by Theorem
  \ref{theorem:boundedweakcontapprox}. Moreover, it satisfies the
  necessary bound \eqref{eq:markovpsibound} for $ \rho $ and
  \eqref{eq:contint} is satisfied by (norm)-continuity of
  $t \mapsto \lambda_t$. Hence all the conditions of Assumption
  \ref{ass:generic} are satisfied and the solution operator therefore
  defines a generalized Feller semigroup $ (P_t) $ on
  $ \mathcal{B}^\rho(\mathcal{E}^w) $ by Theorem
  \ref{theorem:strongcontprocess}.
\end{proof}

We need an additional assumption assuring that the above defined state
space contains the cone hull of $ \mathcal{S}^*_t \nu $:
\begin{assumption}\label{ass:crucial_abstract}
  Let $ \nu $ be such that $ \mathcal{S}^*_u \nu \in \mathcal{E}^w $
  for all $ u > 0 $ and for all $ w > 0 $.
\end{assumption}

\begin{remark}\label{rem:crucial_abstract}
  This condition is satisfied if $K$ and $R^w$ are nonnegative for all
  $w >0$.  Indeed, $ \mathcal{S}^*_u \nu \in \mathcal{E}^w$ if and
  only if
  \begin{align}\label{eq:condres}
    &\langle g, \mathcal{S}_{t+u}^*\nu \rangle -\int_0^t R^w(t-s)\langle g, \mathcal{S}_{s+u}^*\nu \rangle ds =K(t+u) -\int_0^t R^w(t-s) K(s+u) ds \geq 0.
  \end{align}
  Since by the resolvent equation we have
  \begin{align*}
    R^w(t+u)&=w(K(t+u)-\int_0^{t+u} R^w(t+u-s)K(s)ds) \\
            &=w(K(t+u)-\int_0^t R^w(t-s)K(s+u)ds - \int_0^u R^w(t+u-s)K(s)ds),
  \end{align*}
  it follows that
  \[
    K(t+u) -\int_0^t R^w(t-s) K(s+u) ds= \frac{1}{w}R^w(t+u)+ \int_0^u
    R^w(t+u-s)K(s)ds.
  \]
  If $K$ and $R^w$ are nonnegative \eqref{eq:condres} is clearly
  satisfied.
\end{remark}

The following lemma states that the scale of invariant spaces
$ \mathcal{E}^w $ is actually decreasing.
\begin{lemma} Let Assumptions \ref{ass:semigroup} and
  \ref{ass:crucial_abstract} be in force.  For $ w_1 < w_2 $ we have
  $ \mathcal{E}^{w_1} \supseteq \mathcal{E}^{w_2} $. Additionally we
  have that for $ \lambda_0 \in \mathcal{E}^{w_2} $ the unique
  solution of
  \[
    \lambda_t = \mathcal{S}^*_t \lambda_0 - w_1 \int_0^t
    \mathcal{S}^*_{t-s} \nu \langle g, \lambda_s \rangle ds
  \]
  for all $ t \geq 0 $ actually lies in $ \mathcal{E}^{w_2} $.
\end{lemma}
\begin{proof}
  Fix $ w_1 < w_2 $ and $ \lambda_0 \in \mathcal{E}^{w_2} $. Actually
  we can write
  \begin{align*}
    \lambda_t = & \mathcal{S}^*_t \lambda_0 - w_1 \int_0^t   \mathcal{S}^*_{t-s} \nu \langle g, \lambda_s \rangle ds  = \\
    = & \underbrace{\mathcal{S}^*_t \lambda_0 - w_2 \int_0^t   \mathcal{S}^*_{t-s} \nu \langle g, \lambda_s \rangle ds}_{A}   + \underbrace{(w_2-w_1) \int_0^t   \mathcal{S}^*_{t-s} \nu \langle g, \lambda_s \rangle ds.}_{B}
  \end{align*}
  By Proposition \ref{prop:invariance}, the term $A$ clearly lies in
  $ \mathcal{E}^{w_2} $. The same holds true for the second term $B$
  due to Assumption \ref{ass:crucial_abstract} and the cone property
  of $\mathcal{E}^{w_2}$, since $(w_2-w_1) >0$ and
  $\langle g, \lambda_s \rangle \geq 0$ for all $s>0$. Thus
  $\lambda_t \in \mathcal{E}^{w_2} $.  However, by definition of
  $ \mathcal{E}^{w_1} $, $\lambda_t$ also lies there, whence the
  conclusion.
\end{proof}

\begin{definition}\label{statespaceE}
  Let Assumptions \ref{ass:semigroup} and \ref{ass:crucial_abstract}
  be in force.  Then we define the following weak-$*$-closed cone in
  $Y^*$
  \[
    \mathcal{E} = \cap_{w > 0} \mathcal{E}^w \, .
  \]
\end{definition}

\begin{theorem}\label{th:existenceintersect}
  Let Assumptions \ref{ass:semigroup} and \ref{ass:crucial_abstract}
  be in force. Then the weak-$*$-closed cone $ \mathcal{E} $ contains
  all damped measures $ \mathcal{S}^*_t \nu $ for $ t > 0 $ and is
  left invariant by the solution of
  \begin{align} \label{eq:lambdadet1} d \lambda_t = \mathcal{A}^*
    \lambda_t dt - w \nu \langle g, \lambda_t\rangle dt
  \end{align}
  for all $ w > 0 $. This solution defines a generalized Feller
  semigroup on $ \mathcal{B}^\rho(\mathcal{E}) $.
\end{theorem}
\begin{proof}
  The first assertion follows from Assumption
  \ref{ass:crucial_abstract}.  By monotonicity of $ \mathcal{E}^w $
  and the fact that solutions with respect to $ w_1 $-equations
  actually lie in $ \mathcal{E}^{w_2} $ as stated in the previous
  lemma, the invariance assertion and generalized Feller property
  follow together with Proposition \ref{prop:invariance}.
\end{proof}

By the previous results we can now construct a generalized Feller
process on $ \mathcal{E} $ which jumps up by multiples of
$ \mathcal{S}^*_{\varepsilon}\nu $ for some $\varepsilon \geq 0$ and
with an instantaneous intensity of size
$ \langle g, \lambda_t\rangle $.  We formulate these assertions in two
propositions. For their formulation recall that
$\mathcal{E}_* \subset Y$ denotes the (pre-)polar cone of
$ \mathcal{E}$. Define also the following set
\begin{align}\label{eq:mathcalD}
  \mathcal{D}=\{ y \in  Y \, |\, y \in \operatorname{dom}(\mathcal{A})~\text{ s.t. } \langle y, \nu \rangle \text{ is well-defined}\}.
\end{align}

\begin{proposition}\label{prop:epsjumps}
  Let Assumptions \ref{ass:semigroup} and \ref{ass:crucial_abstract}
  be in force. Moreover, let $\mu \in \mathcal{M}_+(\mathbb{R}_{+})$
  with finite second moment.  Consider the SPDE
  \begin{align}\label{eq:SPDEbase}
    d \lambda_t & = \mathcal{A}^* \lambda_t  dt - w \nu  \langle g, \lambda_t \rangle dt + \mathcal{S}^*_\varepsilon \nu dN_t ,
  \end{align}
  where $(N_t)_{t \geq 0}$ is a pure jump process with compensator
  $F(\lambda_t, d\xi)=\langle g,\lambda_t\rangle \mu(d\xi)$.
  \begin{enumerate}
  \item Then for every $ \lambda_0 \in \mathcal{E} $,
    $ \varepsilon > 0 $ and $ w > 0 $, the SPDE \eqref{eq:SPDEbase}
    has a solution in $ \mathcal{E} $ given by a generalized Feller
    process associated to the generator of \eqref{eq:SPDEbase}.
  \item This generalized Feller process is \emph{also} a
    probabilistically weak and analytically mild solution of
    \eqref{eq:SPDEbase}, i.e.
    \[
      \lambda_t = \mathcal{S}^*_t \lambda_0 -\int_0^t
      \mathcal{S}^*_{t-s} w \nu \langle g, \lambda_s \rangle ds +
      \int_0^t\mathcal{S}^*_{t-s+\varepsilon} \nu dN_s \, ,
    \]
    which justifies Equation \eqref{eq:SPDEbase}, in particular for
    every initial value the process $ N $ can be constructed on an
    appropriate probabilistic basis. The stochastic integral is
    defined in a pathwise way along finite variation paths. Moreover,
    for every family $(f_n)_n \in \mathcal{B}^{\rho}(\mathcal{E})$,
    $t \mapsto f_n(\lambda_t)$ can be chosen to be c\`agl\`ad for all
    $n$.
  \item For every $ \varepsilon > 0 $ and $ w > 0 $, the corresponding
    Riccati equation $\partial_t y_t=R(y_t)$ with
    $R: \mathcal{D} \cap \mathcal{E}_* \to \mathbb{R}$ given by
    \begin{align}\label{eq:Riccatisimpl}
      R(y) = \mathcal{A} y - wg \langle y, \nu\rangle +g\int_{\mathbb{R}_+}
      \left(\exp(  \langle y , \mathcal{S}^*_{\varepsilon} \nu \xi \rangle )-1  \right) \mu(d\xi),
    \end{align}
    admits a unique global solution in the mild sense for all initial
    values $ y_0 \in \mathcal{E}_* $.
  \item The affine transform formula holds true.
    \[
      \mathbb{E}_{\lambda_0}\left[ \exp( \langle y_0, \lambda_t
        \rangle)\right]=\exp(\langle y_t, \lambda_0 \rangle),
    \]
    where $y_t$ solves $\partial_t y_t=R(y_t)$ for all
    $y_0 \in \mathcal{E}_*$ in the mild sense with $R$ given by
    \eqref{eq:Riccatisimpl}. Moreover $y_t \in \mathcal{E}_*$ for all
    $t \geq 0$.
  \end{enumerate}
\end{proposition}

\begin{proof}
  To prove the first assertion we apply Proposition
  \ref{prop:jump_perturbation}.  By Theorem
  \ref{th:existenceintersect}, the deterministic equation
  \eqref{eq:lambdadet1} has a mild solution on $\mathcal{E}$ which --
  by Assumption \ref{ass:semigroup} -- defines a generalized Feller
  semigroup $(P_t)_{t\geq 0}$ on $ \mathcal{B}^\rho(\mathcal{E}) $.
  The operator $A$ in Proposition \ref{prop:jump_perturbation} then
  corresponds to the generator of $(P_t)_{t\geq 0}$, i.e.~the
  semigroup associated to the purely deterministic part of
  \eqref{eq:SPDEbase}, which is clearly of transport type and with
  growth bound $M_1 \exp(\omega t) $ for some $M_1\geq 1$ and
  $\omega$.

  Note that by the same arguments as in Proposition
  \ref{prop:invariance} and by applying Theorem
  \ref{theorem:strongcontprocess}, we can prove that
  $(P_t)_{t \geq 0}$ also defines a generalized Feller semigroup on
  $ \mathcal{B}^{\sqrt{\rho}}(\mathcal{E}) $. Indeed, by
  weak-$*$-continuity, we can conclude that
  $\lambda_0 \mapsto f(\lambda_t)$ lies in
  $ \mathcal{B}^{\sqrt{\rho}}(\mathcal{E}) $ for a dense set of
  $ \mathcal{B}^{\sqrt{\rho}}(\mathcal{E})$ by Theorem
  \ref{theorem:boundedweakcontapprox}, whence condition (iii) of
  Assumption \ref{ass:generic} is satisfied. Moreover,
  \eqref{eq:contint} is satisfied by (norm)-continuity of
  $t \mapsto \lambda_t$ and the necessary bound
  \eqref{eq:markovpsibound} holds also for $ \sqrt{\rho} $.  For the
  latter observe first that
  \[
    \| \mathcal{S}^*_t \lambda\|_{Y^*}\leq C \|\lambda\|_{Y^*} \quad
    \text{for all } 0 \leq t \leq T \text{ and } \lambda \in Y^*,
  \]
  since $ \mathcal{S}^* $ is strongly continuous on $ Y^* $, for some
  $ T > 0 $.  Using this we can estimate
  \begin{align}\label{eq:estimatenormlambda}
    \|\lambda_t\|_{Y^*}  \leq 
    C \|\lambda_0\|_{Y^*} + \int_0^t  w \|\mathcal{S}^*_{t-s}\nu\|_{Y^*} \|g\|_Y \|\lambda_s\|_{Y^*}ds.
  \end{align}
  We now proceed similarly as in Proposition \ref{prop:existenceY*}.
  Consider the kernel
  $K'(t,s)=w \|\mathcal{S}^*_{t-s} \nu\|_{Y^*} \|g\|_Y1_{\{s \leq
    t\}}$, which is again a Volterra kernel in the sense of
  \cite[Definition 9.2.1]{gri:90}. Moreover, for any interval
  $[U,V] \subset \mathbb{R}_+$ we have by Young's convolution
  inequality
  \[
    ||| K'|||_{L^1([U,V])} \leq w \|g\|_Y
    \int_0^{V-U}\|\mathcal{S}^*_{s} \nu\|_{Y^*} ds \leq (V-U)w
    \|g\|^2_Y (\int_0^{V-U}\|\mathcal{S}^*_{s} \nu\|^2_{Y^*}
    ds)^{\frac{1}{2}},
  \]
  where $||| \cdot|||_{L^1([U,V])}$ is defined in \cite[Definition
  9.2.2]{gri:90}. We now can again literally take the proof of
  \cite[Lemma 3.1]{AbiLarPul:17} to deduce that the generalized
  Gronwall Lemma (see \cite[Lemma 9.8.2]{gri:90}) can be applied. This
  yields for $t \in [0,T]$
  \[
    \|\lambda_t\|_{Y^*} \leq \|\lambda_0\|_{Y^*} C(1-\int_0^t R'(s)
    ds),
  \]
  where $R'$ denotes the resolvent of $-K'$, which is
  nonpositive. Hence, for $t \in [0,T]$ we have by Jensen's inequality
  \begin{align*}
    \sqrt{\rho(\lambda_t)} =\sqrt{1+ \| \lambda_t\|_{Y^*}^2} \leq 1+ \| \lambda_t\|_{Y^*}\leq 1+ \|\lambda_0\|(C-\int_0^T R'(s) ds)\\
    \leq \tilde C\sqrt{1+\|\lambda_0\|_{Y^*}^2}= \tilde C\sqrt{\rho(\lambda_0)},
  \end{align*}
  where $\tilde C$ depends on $T$.
	
  Finally, we need to verify \eqref{eq:cond1} - \eqref{eq:cond3},
  which read as follows
  \begin{align*}
    \int \rho(\lambda + \mathcal{S}^*_{\varepsilon} \nu \xi) \langle g, \lambda \rangle \mu(d\xi) &\leq M\rho(\lambda)^2, \\
    \int \sqrt{\rho}(\lambda + \mathcal{S}^*_{\varepsilon} \nu \xi) \langle g, \lambda \rangle \mu(d\xi) &\leq M \rho(\lambda),\\
    \int \langle g, \lambda \rangle \mu(d\xi) &\leq M \sqrt{\rho(\lambda)}. 
  \end{align*}
  which hold true by the second moment condition on $\mu$.  Concerning
  \eqref{eq:cond4}, denote as in Remark \ref{rem:quasicontractive}
  \[
    \tilde{\rho}(\lambda) =\sup_{t\geq 0} \exp(-\omega t) P_t
    \rho(\lambda) \, .
  \]
  In particular we know that $ \rho \leq \tilde \rho$ and it holds
  that $P_tf(x)=f(\psi_t(x))$ where $\psi$ is the solution of
  \eqref{eq:lambdadet1} which is linear. Using this together with
  $ | \sup_t c(t) - \sup_t d(t)| \leq \sup_t |c(t) - d(t) | $ we
  obtain similarly as in the proof of Corollary \ref{cor:affine}
  \begin{align*}
    & \int  \big | \frac{\tilde{\rho}(\lambda + \mathcal{S}^*_{\varepsilon} \nu \xi)-\tilde{\rho}(\lambda)}{\tilde{\rho}(\lambda)} \big | \langle g, \lambda \rangle\mu(d\xi)  \\ &\leq\int \big| \frac{\sup_{t \geq 0} \exp(-\omega t) |P_t\rho(\lambda + \mathcal{S}^*_{\varepsilon} \nu \xi)) -  P_t \rho( \lambda)|}{\tilde\rho(\lambda)} \big |  \langle g, \lambda \rangle\mu(d\xi)\\      
    &\leq\int \big| \frac{\sup_{t \geq 0} \exp(-\omega t) |\rho(\psi_t(\lambda + \mathcal{S}^*_{\varepsilon} \nu \xi))) -  \rho(\psi_t (\lambda))|}{\tilde\rho(\lambda)} \big |  \langle g, \lambda \rangle\mu(d\xi)\\   
    &= \int \big | \frac{ \sup_{t \geq 0} \exp(-\omega t) ( 2\| \psi_t( \lambda) \|_{Y^*} \; \| \psi_t( \mathcal{S}^*_{\varepsilon} \nu \xi )\|_{Y^*} + {\| \psi_t (\mathcal{S}^*_{\varepsilon} \nu \xi) \|}_{Y^*}^2 )}{\rho( \lambda )} \big | \langle g, \lambda \rangle\mu(d\xi)  \leq \widetilde{\omega} \, .
  \end{align*}
  for some $\widetilde{\omega} \in \mathbb{R}$. The last inequality
  holds by the linearity of $\psi$ and the second moment condition on
  $\mu$.  Proposition \ref{prop:jump_perturbation} now allows to
  conclude that $A+B$ where $B$ is given by
  \[
    Bf(\lambda)= \int (f(\lambda + \mathcal{S}^*_{\varepsilon} \nu
    \xi) -f(\lambda)) \langle g, \lambda \rangle\mu(d\xi)
  \]
  generates a generalized Feller semigroup $\widetilde P$ as asserted.

  For (ii) we now construct the probabilistically weak and
  analytically mild solution directly from the properties of the
  generalized Feller process: take $ y \in \mathcal{D} $ defined in
  \eqref{eq:mathcalD} and consider the martingale
  \begin{equation}\label{eq:mart}
    \begin{split}
      M^y_t & := \langle y, \lambda_t \rangle - \langle y , \lambda_0 \rangle - \int_0^t \langle \mathcal{A} y, \lambda_s \rangle - w\langle y ,  \nu  \rangle \langle g , \lambda_s \rangle ds  \\
      & \quad - \int_0^t \int \langle y, \mathcal{S}^*_\varepsilon \nu
      \xi \rangle \langle g, \lambda_s \rangle \mu(d \xi) ds
    \end{split}
  \end{equation}
  for $ t \geq 0 $ (after appropriate regularization as it is possible
  due to Theorem \ref{th:path_properties} such that the integral term
  is well defined). Let now $y$ be as above with the additional
  property that
  $ \langle y , \mathcal{S}^*_{\varepsilon}\nu \rangle = 1 $. Define
  \begin{align}\label{eq:N}
    N_t := M^y_t + \int_0^t \int \langle y, \mathcal{S}^*_\varepsilon \nu \xi
    \rangle \langle g, \lambda_s \rangle \mu(d \xi) ds
  \end{align}
  for $ t \geq 0 $, which is a c\`agl\`ad semimartingale. Then $N$
  does not depend on $y$. Indeed, for all $y_i$ with
  $ \langle y_i , \mathcal{S}^*_{\varepsilon}\nu \rangle = 1 $,
  $i=1,2$, we clearly have
  \[
    \int_0^t \int \langle y_1-y_2, \mathcal{S}^*_\varepsilon \nu
    \xi\rangle \langle g, \lambda_s \rangle \mu(d \xi) ds =0
  \]
  and $ M^{y_1} - M^{y_2} =M^{y_1-y_2}= 0 $ and as well.  The latter
  follows from the fact that the martingale $ M^y $ is constant if
  $ \langle y , \mathcal{S}^*_\varepsilon \nu \rangle = 0 $. This is
  indeed true since the martingale's quadratic variation vanishes in
  this case: we apply here the carr\'e du champ formula
  \[
    \mathbb{E}\big[[M^y_t,M^y_t] \big] = \mathbb{E} \big[ \int_0^t
    \big ( A f^2 (\lambda_s)- 2 f(\lambda_s) A f(\lambda_s) \big) ds
    \big]
  \]
  where $ f(\lambda) = \langle y , \lambda \rangle $ with $y$
  satisfying $\langle y, \mathcal{S}^*_\varepsilon \nu \rangle =
  0$. Moreover, by the definition of $N$ in \eqref{eq:N} its
  compensator is given by
  $\int_0^t \int \xi \langle \lambda_s,g \rangle \mu(d\xi)$ showing
  that $N$ has the desired properties.

  By \eqref{eq:mart} and \eqref{eq:N} we additionally obtain that
  \begin{align*}
    \langle y, \lambda_t \rangle & = \langle y , \lambda_0 \rangle + \int_0^t \langle \mathcal{A}y , \lambda_s \rangle ds - \int_0^t \langle y, w \nu \rangle \langle g , \lambda_s \rangle  ds + \langle y , \mathcal{S}^*_\varepsilon \nu \rangle N_t \\
  \end{align*}
  for $ y \in \mathcal{D} $. This analytically weak form can be
  translated into a mild form by standard methods. Indeed, notice that
  the integral is just along a finite variation path and therefore we
  can readily apply variation of constants. The last assertion about
  the c\`agl\`ad property is a consequence of Theorem
  \ref{th:path_properties} by noting that $\rho(\lambda)$ does not
  explode. This proves (ii).
		 
  Concerning (iii), note first that we have a unique mild solution to
  \begin{align}\label{eq:linadjoint}
    \partial_t y_t = \mathcal{A}y_t-w g \langle y_t , \nu \rangle, \quad y_0 \in Y,
  \end{align}
  since this is the adjoint equation of
  \eqref{eq:lambda_abstract}. For the equation with jumps we proceed
  as in Proposition \ref{prop:existenceY*} via Picard
  iteration. Denote the semigroup associated to \eqref{eq:linadjoint}
  by $\mathcal{S}^w$ and define
  \begin{align*}
    y_t^0&= y_0,\\
    y_t^{n}&=\mathcal{S}^w_ty_0+\int_0^t \mathcal{S}^w_{t-s}g\left(\int_{\mathbb{R}_+}
             \left(\exp(  \langle y^{n-1}_s , \mathcal{S}^*_{\varepsilon} \nu \xi \rangle )-1  \right) \mu(d\xi) \right)ds.
  \end{align*}
  Moreover, for $t \in [0,\delta]$ for some $\delta >0$ we have by
  local Lipschitz continuity of $x\mapsto \exp(x)$
  \begin{align*}
    \| y_t^{n+1}-y_t^n\|_{Y} &\leq  \|\int_0^t \mathcal{S}^w_{t-s}g \left(\exp(  \langle y^{n}_s , \mathcal{S}^*_{\varepsilon} \nu \xi \rangle )- \exp(  \langle y^{n-1}_s , \mathcal{S}^*_{\varepsilon} \nu \xi \rangle ) \right) \mu(d\xi) ds\|_{Y}\\
                             & \leq \int_0^t C \|\mathcal{S}^w_{t-s}g\|_Y \| y_t^{n}-y_t^{n-1}\|_{Y} \left(\int \|\mathcal{S}^*_{\varepsilon} \nu \xi\|_{Y*}\mu(d\xi) \right) ds.
  \end{align*}
  By an extension of Gronwall's inequality (see \cite[Lemma 15]{D:99})
  this yields convergence of $(y_t^n)_{n\in \mathbb{N}}$ with respect
  to $\|\cdot \|_{Y}$ and hence the existence of a unique local mild
  solution to \eqref{eq:Riccatisimpl} up to some maximal life time
  $t_+(y_0)$.  That $t_+(y_0)= \infty$ for all $y_0 \in \mathcal{E}_*$
  follows from the subsequent estimate
  \begin{align*}
    \| y_t\|_{Y} &= \|\mathcal{S}^w_t y_0 +\int_0^t \mathcal{S}^w_{t-s} g \left(\int \left(\exp(  \langle y_s , \mathcal{S}^*_{\varepsilon} \nu \xi \rangle )-1 \right) \mu(d\xi) \right)ds \|_{Y}\\
                 &\leq \|\mathcal{S}^w_t y_0 \|_Y +  \int_0^t \|\mathcal{S}^w_{t-s} g\|_Y \left(\int |\exp(\langle y_s , \mathcal{S}^*_{\varepsilon} \nu \xi \rangle )-1| \mu(dx) \right)ds\\
                 &\leq \|\mathcal{S}^w_t y_0 \|_Y  +  t \sup_{s \leq t}\|\mathcal{S}^w_{s} g\|_Y \mu(\mathbb{R}_+),
  \end{align*}
  where we used
  $|\exp(\langle y , \mathcal{S}^*_{\varepsilon} \nu \xi \rangle )-1|
  \leq 1$ for all $y \in \mathcal{E}_*$ in the last estimate.
		
  To prove (iv), just note that by the existence of a generalized
  Feller semigroup the abstract Cauchy problem for initial value
  $ \exp(\langle y_0,.\rangle) $ can be solved uniquely for
  $ y_0 \in \mathcal{E}_* $. Indeed,
  $\mathbb{E}_{\lambda} [\exp(\langle y_0,\lambda_t\rangle)]$ uniquely
  solves
  \[
    \partial_t u(t,\lambda)= A u(t, \lambda), \quad u(0, \lambda) =
    \exp(\langle y_0, \lambda \rangle),
  \]
  where $A$ denotes the generator associated to \eqref{eq:SPDEbase}.
  Setting $u(t,\lambda)=\exp(\langle y_t, \lambda\rangle )$, we have
  \[
    \partial_t u(t,\lambda)=\exp(\langle y_t, \lambda \rangle )
    R(y_t),
  \]
  where the right hand side is nothing else than
  $A\exp(\langle y_t, \lambda \rangle)$, hence the affine transform
  formula holds true. This also implies that $y_t \in \mathcal{E}_*$
  for all $t \geq 0$, simply because
  $\mathbb{E}_{\lambda} [\exp(\langle y_0,\lambda_t\rangle)] \leq 1$
  for all $\lambda \in \mathcal{E}$.
\end{proof}

The next statement is a refinement of Conclusion (i) of Proposition
\ref{prop:epsjumps} for the case $\varepsilon =0$. All other
conclusions are explicitely stated in Theorem \ref{thm:main} below.

\begin{proposition}\label{prop:jumpseps0}
  Let Assumptions \ref{ass:semigroup} and \ref{ass:crucial_abstract}
  be in force. Moreover, let $\mu \in \mathcal{M}_+(\mathbb{R}_{+})$
  with finite second moment.  Then Conclusion (i) of Proposition
  \ref{prop:epsjumps} holds true for $\varepsilon=0$.  Moreover, the
  affine transform formula for the generalized Feller process
  $(\lambda_t)_{t\geq 0}$ of \eqref{eq:SPDEbase} with $\varepsilon=0$
  is also satisfied, i.e.
  \[
    \mathbb{E}_{\lambda_0}\left[ \exp( \langle y_0, \lambda_t
      \rangle)\right]=\exp(\langle y_t, \lambda_0 \rangle),
  \]
  where $y_t$ solves $\partial_t y_t=R(y_t)$ for all
  $y_0 \in \mathcal{E}_*$ in the mild sense with $R$ given by
  \eqref{eq:Riccatisimpl} with $\varepsilon=0$. Furthermore,
  $y_t \in \mathcal{E}_*$ for all $t \geq 0$.
\end{proposition}
	
	\begin{remark}
          The essential point here is that we loose the c\`agl\`ad
          property when we let $ \varepsilon $ tend to zero. As long
          as the kernel $ K $ has a singularity at $ t = 0 $ it is
          impossible to preserve finite growth bounds \emph{without}
          $ M > 1 $ as $ \varepsilon \to 0 $. To apply the second
          conclusion of Theorem \ref{th:path_properties} yielding
          c\`agl\`ad trajectories it is however crucial that $M=1$.
        \end{remark}

\begin{proof}
  We apply Theorem \ref{thm:approximation_gen} and consider a sequence
  of generalized Feller semigroups $(P^n)_{n \in \mathbb{N}}$ with
  generators $A^n$ corresponding to the solution of
  \eqref{eq:SPDEbase} for $\varepsilon =\frac{1}{n}$,
  $n \in \mathbb{N}$. Let us first establish a uniform growth bound
  for this sequence. Note that for the solution of
  \eqref{eq:SPDEbase}, we have due to Proposition \ref{prop:epsjumps}
  (ii) the following estimate for $t \in [0,T]$ for some fixed $T>0$
  \begin{align*}
    \mathbb{E}[\| \lambda^{\varepsilon}_t\|^2_{Y^*}] &\leq 3\|\mathcal{S}^*_t \lambda_0\|^2_{Y^*}+ 3t\int_0^t w^2\|\mathcal{S}^*_{t-s}\nu \|^2_{Y^*} \|g\|^2_{Y} \mathbb{E}[\|\lambda^{\varepsilon}_s\|^2_{Y^*}]ds \\
                                                     &\quad + 6\mathbb{E} [\|\int_0^t \mathcal{S}^*_{t-s+\varepsilon} \nu  dN_s -\int_0^t \int \mathcal{S}^*_{t-s+\varepsilon} \nu \xi \mu(d\xi) \langle g, \lambda^{\varepsilon}_s \rangle ds\|^2]\\
                                                     &\quad + 6\mathbb{E} [\|\int_0^t \int \mathcal{S}^*_{t-s+\varepsilon} \nu \xi \mu(d\xi) \langle g, \lambda^{\varepsilon}_s \rangle ds\|^2]\\
                                                     &\leq C_0 \|\lambda_0\|^2_{Y^*}  + 3t\int_0^t w^2\|\mathcal{S}^*_{t-s}\nu \|^2_{Y^*} \|g\|^2_{Y} \mathbb{E}[\|\lambda^{\varepsilon}_s\|^2_{Y^*}]ds\\
                                                     &\quad +6\mathbb{E} [\int \xi^2\mu(d\xi)\int_0^t  \|\mathcal{S}^*_{t-s+\varepsilon} \nu \|_{Y^*}^2\langle g, \lambda^{\varepsilon}_s \rangle  ds]\\
                                                     &\quad +  C_1 \int \xi^2\mu(d\xi) \int_0^t \|\mathcal{S}^*_{t-s+\varepsilon} \nu\|^2_{Y^*}\|g\|^2_Y\mathbb{E}[ \|\lambda^{\varepsilon}_s\|^2_{Y^*}] ds\\
                                                     &\leq C_0  \|\lambda_0\|^2_{Y^*}  + 3t\int_0^t w^2\|\mathcal{S}^*_{t-s}\nu \|^2_{Y^*} \|g\|_{Y} \mathbb{E}[\|\lambda^{\varepsilon}_s\|^2_{Y^*}]ds\\
                                                     &\quad +C_2\int \xi^2\mu(d\xi)\int_0^t  \|\mathcal{S}^*_{t-s+\varepsilon} \nu \|_{Y^*}^2  ds\\
                                                     &\quad + C_1 \int \xi^2\mu(d\xi) \int_0^t \|\mathcal{S}^*_{t-s+\varepsilon} \nu\|^2_{Y^*}\|g\|^2_Y\mathbb{E}[ \|\lambda^{\varepsilon}_s\|^2_{Y^*}] ds\\
                                                     &\leq C_0  \|\lambda_0\|^2_{Y^*} + C_2\int_0^t  \|\mathcal{S}^*_{t-s} \nu \|_{Y^*}^2 ds + C_1 \int_0^t \|\mathcal{S}^*_{t-s} \nu\|^2_{Y^*}\mathbb{E}[ \|\lambda^{\varepsilon}_s\|^2_{Y^*}] ds, 
  \end{align*}
  where the constants can change from line to line and depend on $T$.
  We use $\|S_t^*\lambda_0\|^2 \leq C_0 \|\lambda_0\|^2$ for
  $t \in [0,T]$, as well as
  $ \|\mathcal{S}^*_{t-s+\varepsilon} \nu\|_{Y^*} \leq \widehat{C}
  \|\mathcal{S}^*_{t-s} \nu\|_{Y^*}$ for some constant $\widehat{C}$
  and all $\varepsilon \in (0,1]$ due to strong continuity. Exactly by
  the same arguments as in the proof of Proposition
  \ref{prop:existenceY*} , we thus obtain for $t \in [0,T]$ for some
  fixed $T$
  \[
    \mathbb{E}[\|\lambda_t\|^2_{Y^*}] \leq
    \widetilde{C}(\|\lambda_0\|^2_{Y^*}+1) (1-\int_0^t R'(s), ds),
  \]
  where $R'$ denotes the resolvent of
  $-C_2\|\mathcal{S}^*_{t-s}\nu \|_{Y^*} $. Hence,
  $\mathbb{E}[\rho(\lambda_t)]\leq C \rho(\lambda_0)$ for
  $t \in [0,T]$. From this the desired uniform growth bound
  $\|P_t\|_{L(\mathcal{B}^{\rho}(\mathcal{E}))} \leq M \exp(\omega t)$
  for some $M \geq 1 $ and $\omega \in \mathbb{R}$ follows.

  For the set $D$ as of Theorem \ref{thm:approximation_gen} we here
  choose Fourier basis elements of the form
  \begin{align}\label{eq:Fourier}
    f_y: \mathcal{E} \to [0,1]; \lambda \mapsto \exp( \langle y , \lambda \rangle) 
  \end{align}
  such that $y \in \mathcal{E}_*$ and
  $\lambda \mapsto \exp( \langle y , \lambda \rangle)$ lies in
  $\cap_{n \geq 1} \operatorname{dom}(A^n)$, whose span is of course
  dense, whence (i) of Theorem \ref{thm:approximation_gen}.  We now
  equip $\operatorname{span}(D)$ with the uniform norm
  $\| \cdot\|_{\infty}$ and verify Condition (ii), i.e.~we check
  \begin{align}\label{eq:CondII}
    \|A^n P^m_u f_y -A^mP^m_u f_y\|_{\rho} \leq \|f_y\|_{\infty}a_{nm}
  \end{align}
  for all $0 \leq u \leq t$ with $a_{nm} \to 0$ as $n,m \to \infty$,
  and possibly depending on $y$ and $t$.  Note that
  \[
    A^nf_y(\lambda)=\langle R^n(y), \lambda\rangle f_y(\lambda),
  \]
  where $R^n$ corresponds to \eqref{eq:Riccatisimpl} for
  $\varepsilon =\frac{1}{n}$. As $P^n$ leaves $D$ invariant for all
  $n \in \mathbb{N}$ by Proposition \ref{prop:epsjumps} (iv), we have
  \begin{align*}
    &\frac{|A^n P^m_u f_y(\lambda) -A^m P^m_u f_y(\lambda)|}{\rho(\lambda)}\\
    &\quad=\frac{f_{y^m_u}(\lambda) \langle g, \lambda\rangle}{\rho(\lambda)}\int_{\mathbb{R}_+} \underbrace{\exp(\langle y^m_u, \mathcal{S}^*_{\frac{1}{m}} \nu\xi \rangle)}_{\leq 1}\underbrace{|\exp(\langle y^m_u, (\mathcal{S}^*_{\frac{1}{n}} -\mathcal{S}^*_{\frac{1}{m}} \nu ) \xi\rangle) -1 |}_{\widetilde{a}_{nm}}\mu(d\xi )
  \end{align*}
  where $y^m_u$ denotes the solution of $\partial_t y^m_u=R^m(y^m_t)$
  at time $u$ with $y_0=y$ and $\widetilde{a}_{nm}$ is chosen
  uniformly for all $u \leq t$ and tends to $0$ as $n,m \to \infty$.
  This is possible since for the chosen initial values $ y $ we obtain
  that $ {\|y^m_u\|}_Z $ is bounded on compact intervals in time
  uniformly in $ m $ also with respect to the $ Z $ norm. Indeed
  \begin{align*}
    \| y_t\|_{Z} &= \|\mathcal{S}^w_t y_0 +\int_0^t \mathcal{S}^w_{t-s} g \left(\int \left(\exp(  \langle y_s , \mathcal{S}^*_{\varepsilon} \nu \xi \rangle )-1 \right) \mu(d\xi) \right)ds \|_{Z}\\
                 &\leq \|\mathcal{S}^w_t y_0 \|_Z +  \int_0^t \|\mathcal{S}^w_{t-s} g\|_Z \left(\int |\exp(\langle y_s , \mathcal{S}^*_{\varepsilon} \nu \xi \rangle )-1| \mu(dx) \right)ds\\
                 &\leq \|\mathcal{S}^w_t y_0 \|_Z  +   \mu(\mathbb{R}_+) \int_0^t \|\mathcal{S}^w_{u} g\|_Z du  \, ,
  \end{align*}
  where $\mathcal{S}^w$ is defined in the proof of Proposition
  \ref{prop:epsjumps}.  From this we infer \eqref{eq:CondII} with
  $a_{nm}= C\widetilde{a}_{nm}$ for some constant $C$. The conditions
  of Theorem \ref{thm:approximation_gen} are therefore satisfied and
  we obtain a generalized Feller semigroup whose generator corresponds
  to $\eqref{eq:SPDEbase}$ with $\varepsilon =0$.

  Concerning the second assertion, the affine transform formula
  follows simply from the convergence of the semigroups $P^n$ as
  asserted in Theorem \ref{thm:approximation_gen} by setting
  $y_t= \lim_{\varepsilon \to 0} y_t^{\varepsilon}$, where
  $y_t^{\varepsilon}$ solves
  $\partial_t y_t^{\varepsilon}=R^{\varepsilon}(y_t^{\varepsilon} )$
  in the mild sense with $R^{\varepsilon}$ given by
  \eqref{eq:Riccatisimpl}. Since $\exp(\langle y_t, \lambda \rangle)$
  is then also the unique solution of the abstract Cauchy problem for
  initial value $ \exp(\langle y_0,\lambda\rangle) $, i.e.  it solves
  \[
    \partial_t u(t,\lambda)= A u(t, \lambda), \quad u(0, \lambda) =
    \exp(\langle y_0, \lambda \rangle),
  \]
  where $A$ denotes the generator associated to \eqref{eq:SPDEbase}
  with $\varepsilon =0$, we infer that $y_t$ satisfies
  $\partial_t y_t=R(y_t)$ with $R$ given by \eqref{eq:Riccatisimpl}
  for $\varepsilon=0$. This is because
  $A\exp(\langle y_t, \lambda \rangle )=\exp(\langle y_t, \lambda
  \rangle)R(y_t)$.
\end{proof}

We are now ready to state our main theorem, namely an existence and
uniqueness result for equations of the type
\begin{align}\label{eq:SPDEmain}
  d \lambda_t = \mathcal{A}^* \lambda_t dt + \nu dX_t,
\end{align}
where $(X_t)_{t \geq 0}$ is an It\^o semimartingale of the form
\begin{align}\label{eq:X}
  X_t = \int_0^t \beta \langle g, \lambda_s \rangle ds+ \int_0^t
  \sigma \sqrt{\langle g, \lambda_t \rangle}dB_t+ \int_0^t \int \xi
  (\mu^X(d\xi,ds) - \langle g, \lambda_s \rangle m(d\xi) ds),
\end{align}
for some Brownian motion $B$ and random jump measure $\mu^X$ with
$\beta, \sigma \in \mathbb{R}$ and $m(d\xi)$ is a L\'evy measure on
$\mathbb{R}_{+}$ admitting a second moment.

\begin{theorem}\label{thm:main}
  Let Assumptions \ref{ass:crucial_abstract} and \ref{ass:semigroup}
  be in force.
  \begin{enumerate}
  \item Then the stochastic partial differential equation
    \eqref{eq:SPDEmain} admits a unique Markovian solution
    $(\lambda_t)_{t\geq0} $ with values in $ \mathcal{E} $ given by a
    generalized Feller semigroup on $ \mathcal{B}^\rho(\mathcal{E}) $
    whose generator takes on the set of Fourier elements
    \begin{align*}
      f_y: \mathcal{E} \to [0,1]; \lambda \mapsto \exp( \langle y , \lambda\rangle)   
    \end{align*}
    for $y \in \mathcal{D} \cap \mathcal{E}_*$ where $\mathcal{D}$ is
    defined in \eqref{eq:mathcalD} the form
    \begin{align}\label{eq:genmain}
      Af_y(\lambda)=f_y(\lambda) (\langle \mathcal{A}y, \lambda \rangle+ \mathcal{R}(\langle y, \nu \rangle) \langle g,\lambda \rangle.
    \end{align}  
    Here, $\mathcal{R}: \mathbb{R}_- \to \mathbb{R}$ is given by
    \begin{align}\label{eq:Rtrad}
      \mathcal{R}(u)=\beta u+ \frac{1}{2} \sigma^2 u^2+\int_{\mathbb{R}_+}
      \left(\exp( u \xi )-1 - u\xi \right) m(d\xi).
    \end{align}
  \item This generalized Feller process allows to construct a
    probabilistically weak and analytically mild solution of
    \eqref{eq:SPDEmain}, i.e. for $ y \in Y $
    \[
      \langle y , \lambda_t \rangle = \langle y, \mathcal{S}_t^*
      \lambda_0 \rangle +\int_0^t \langle y,\mathcal{S}_{t-s}^* \nu
      \rangle dX_s \, ,
    \]
    in particular for every initial value
    $ \lambda_0 \in \mathcal{E} $ the semimartingale $X$ can be
    constructed on an appropriate probabilistic basis.  The stochastic
    integral is only defined in the \emph{weak} sense, that is after
    pairing with $ y \in Y $, a solution concept which is sometimes
    denoted by \emph{weakly mild}. For every $ y \in Y $ there exists
    a c\`ag version of the real valued process
    $ {(\langle y , \lambda_t \rangle )}_{t \geq 0} $.
  \item The affine transform formula is satisfied, i.e.
    \[
      \mathbb{E}_{\lambda_0}\left[ \exp( \langle y_0, \lambda_t
        \rangle)\right]=\exp(\langle y_t, \lambda_0 \rangle),
    \]
    where $y_t$ solves $\partial_t y_t=R(y_t)$ for all
    $y_0 \in \mathcal{E}_*$ and $t >0$ in the mild sense with
    $R: \mathcal{D} \cap \mathcal{E}_* \to \mathbb{R}$ given by
    \begin{align}\label{eq:Ricfin}
      R(y) = \mathcal{A} y + \mathcal{R}( \langle y , \nu \rangle) g
    \end{align}
    and $\mathcal{R}$ defined in \eqref{eq:Rtrad}.  Furthermore,
    $y_t \in \mathcal{E}_*$ for all $t \geq 0$.
  \item For all $\lambda_0 \in \mathcal{E}$, the corresponding jump
    diffusion stochastic Volterra equation,
    $V_t:= \langle g, \lambda_t \rangle$, given by
    \begin{align}\label{eq:Voltrep}
      V_t= \langle g, \lambda_t \rangle =\langle g, \mathcal{S}_t^* \lambda_0 \rangle +
      \int_0^t\langle g,\mathcal{S}_{t-s}^*\nu\rangle dX_s= h(t) +
      \int_0^t K(t-s)dX_s
    \end{align}
    with $h(t)=\langle g, \mathcal{S}_t^* \lambda_0 \rangle$ admits a
    probabilistically weak solution with c\`ag trajectories.
  \item[(v)] The Laplace transform of the Volterra equation $V_t$ is
    given by
    \begin{align}\label{eq:Voltlaplace}
      \mathbb{E}_{\lambda_0}\left[\exp(u V_t)\right]=\exp(u h(t)+ \int_0^t h(t-s) \mathcal{R}(\psi_s) ds),
    \end{align}
    where $h(t)=\langle g, \mathcal{S}_t^* \lambda_0 \rangle$ and
    $\psi_t:=\langle y_t, \nu\rangle$ solves the Riccati Volterra
    equation
    \[
      \psi_t=u K(t)+\int_0^t K(t-s) \mathcal{R}(\psi_s) ds, \quad t >0
      \, .
    \]
    Hence the solution of the stochastic Volterra equation in
    \eqref{eq:Voltrep} is unique in law.
  \end{enumerate}
\end{theorem}

\begin{remark}
  \begin{enumerate}
  \item Observe that \eqref{eq:Voltrep} implies that the function
    $h(t)$ defined by
    $h(t) = \langle g, \mathcal{S}^*_t \lambda_0 \rangle$ is given by
    the expectation
    \[
      h(t) = \mathbb{E}[V_t - \beta \int_0^t K(t-s) V_s ds].
    \]
    In the diffusion case, \eqref{eq:Voltlaplace} therefore coincides
    with the formula obtained in \cite{AbiLarPul:17}.  From the Markov
    property of $(\lambda_t)_{t \geq0}$, it is easy to see that the
    conditional Laplace transform of $V_t$ reads as
    \begin{align*}
      \mathbb{E}_{\lambda_0}\left[\exp(u V_t)\,|\,
      \mathcal{F}_s\right]&= \exp(\langle y_{t-s}, \lambda_s\rangle)
      \\
                          &= \exp(u \langle g, \mathcal{S}_{t-s}^* \lambda_s \rangle+
                            \int_0^{t-s}  \langle g, \mathcal{S}^*_{t-s-r} \lambda_u \rangle
                            \mathcal{R}(\langle y_r, \nu \rangle) dr),
    \end{align*}
    where $y_t$ is the solution of $\partial_t y_t=R(y_t)$ with $R$
    given by \eqref{eq:Ricfin} and $y_0=u g$.  Here,
    $\langle g, \mathcal{S}^*_{x} \lambda_s \rangle$ corresponds to
    the following conditional expectation
    \[
      \langle g, \mathcal{S}^*_{x} \lambda_s \rangle
      =\mathbb{E}\left[V_{s+x} - \beta \int_0^x K(x-r) V_{s+r} dr \,
        |\, \mathcal{F}_s \right], \quad x \geq 0.
    \]
    The function-valued process
    $(x \mapsto \langle g, \mathcal{S}^*_{x} \lambda_t \rangle)_{t
      \geq 0}$ corresponds also to a Markovian lift, described in
    detail in Section \ref{sec:markovianlift_forward}. Note that
    assuming existence of $V$, \eqref{eq:Voltlaplace} and also the
    conditional Laplace transform of $V$ is completely independent of
    the kernel representation.

  \item Mimicking the proof of assertion (v) of the above theorem, one
    can show that
    \[
      \mathbb{E}_{\lambda_0}\left[ \exp( \langle y_0, \lambda_t
        \rangle)\right]=\exp(\langle y_t, \lambda_0 \rangle)=
      \exp(\langle y_0, S_t^* \lambda_0 \rangle + \int_0^t h(t-s)
      \mathcal{R}(\widetilde{\psi}_s) ds),
    \]
    where $\widetilde{\psi}_t$ solves
    \[
      \widetilde{\psi}_t= \langle y_0, \mathcal{S}^*_t \nu \rangle +
      \int_0^t K(t-s) \mathcal{R}(\widetilde{\psi}_s) ds.
    \]
  \end{enumerate}
\end{remark}

\begin{proof}
  To prove (i), we proceed again by applying Theorem
  \ref{thm:approximation_gen}. To this end consider the following
  sequence of SPDEs
  \begin{equation}\label{eq:approxSPDE}
    \begin{split}
      d \lambda^n_t  &= \mathcal{A}^* \lambda_t  dt +\beta \nu  \langle g, \lambda^n_t \rangle dt - n \sigma^2 \nu  \langle g, \lambda^n_t \rangle dt + \nu dN^n_{1,t}\\
      &\quad - \int_{\{\xi > \frac{1}{n}\}} \xi m(d\xi) \nu \langle g,
      \lambda^n_t \rangle dt + \nu dN^n_{2,t},
    \end{split}
  \end{equation}
  where $N^n_{t,1}$ is a jump process that jumps by $1/n$ and
  intensity $\sigma^2 n^2 \langle g, \lambda^n_t\rangle$, i.e. its
  compensator is given by
  $\sigma^2 n^2 \langle g, \lambda^n_t
  \rangle\delta_{\frac{1}{n}}(d\xi)$. The compensator of $N^n_{t,2}$
  is given by
  $1_{\{\xi > \frac{1}{n}\}} m(d\xi) \langle g, \lambda^n_t
  \rangle$. In terms of \eqref{eq:SPDEbase}, this corresponds to
  $\varepsilon =0$, $N=N^n=N^n_1+N^n_2$,
  $w=-\beta+n\sigma^2+\int_{\{\xi > \frac{1}{n}\}} \xi m(d\xi)$ and
  $\mu^n(d\xi):=\sigma^2 n^2\delta_{\frac{1}{n}}(d\xi)+ 1_{\{\xi >
    \frac{1}{n}\}}m(d\xi)$. From Proposition \ref{prop:jumpseps0} we
  thus know that these SPDEs admit a solution in terms of generalized
  Feller semigroups $P^n$ with associated generators $A^n$.  Let us
  now establish a uniform growth bound for this sequence. Due to
  Proposition \ref{prop:jumpseps0} we can establish the following
  estimate for $t \in [0,T]$ for some fixed $T>0$
  \begin{align*}
    \mathbb{E}[\| \lambda^{n}_t\|^2_{Y^*}] &\leq 3\|\mathcal{S}^*_t \lambda_0\|^2_{Y^*}+ 
                                             3t\int_0^t \beta^2\|\mathcal{S}^*_{t-s}\nu \|^2_{Y^*} \|g\|^2_{Y} \mathbb{E}[\|\lambda^{n}_s\|^2_{Y^*}]ds \\   
                                           &\quad +3\mathbb{E} [\|\int_0^t \mathcal{S}^*_{t-s} \nu  dN^n_s -\int_0^t \int \mathcal{S}^*_{t-s} \nu \xi \mu^n(d\xi) \langle g, \lambda^{n}_s \rangle ds\|^2]\\
                                           &\leq 3C_0\|\lambda_0\|^2_{Y^*}+3t\int_0^t\beta^2  \|\mathcal{S}^*_{t-s}\nu \|^2_{Y^*} \|g\|^2_{Y} \mathbb{E}[\|\lambda^{n}_s\|^2_{Y^*}]ds\\
                                           &\quad +
                                             3\mathbb{E} [\int \xi^2\mu^n(d\xi)\int_0^t  \|\mathcal{S}^*_{t-s} \nu \|_{Y^*}^2\langle g, \lambda^{n}_s \rangle  ds]\\
                                           &\leq 3C_0\|\lambda_0\|^2_{Y^*}+3(\sigma^2+\int \xi^2 m(d\xi)) \int_0^t  \|\mathcal{S}^*_{t-s} \nu \|_{Y^*}^2 ds\\
                                           &\quad+3\|g\|^2_Y (\beta^2t+\sigma^2+\int \xi^2 m(d\xi))\int_0^t \|\mathcal{S}^*_{t-s} \nu\|^2_{Y^*}\mathbb{E}[ \|\lambda^{n}_s\|^2_{Y^*}] ds.
  \end{align*}
  Exactly as in Proposition \ref{prop:jumpseps0}, this now yields the
  uniform growth bound.  For the set $D$ as of Theorem
  \ref{thm:approximation_gen} we consider again the Fourier basis
  elements as in \eqref{eq:Fourier}, whose span is of course dense,
  whence (i).  We equip $\operatorname{span}(D)$ again with the
  uniform norm $\| \cdot\|_{\infty}$ and obtain a contraction property
  there. To verify Condition (ii) of Theorem
  \ref{thm:approximation_gen}, we check \eqref{eq:CondII} in the
  present setting. Note that $A^n$ takes the form
  \[
    A^nf_y(\lambda)=f_y(\lambda) (\langle \mathcal{A}y, \lambda
    \rangle+\beta \langle g, \lambda \rangle \langle y, \nu \rangle +
    \langle g, \lambda \rangle \int (\exp(\langle y, \nu \xi \rangle)
    -1 -\langle y, \nu \xi \rangle) \mu^n(d\xi))
  \]
  on $D$ and from the last assertion of Proposition
  \ref{prop:jumpseps0} we know that $P^n$ leaves $D$ invariant for all
  $n \in \mathbb{N}$. By this invariance it thus suffices to check
  \begin{align}\label{eq:condIImod}
    \|A^n f_y -A^mf_y\|_{\rho} \leq \|f_y\|_{\infty}a_{nm}
  \end{align}
  for $ y \in K $, where $ K $ is a bounded set with respect to the
  $ Z $ norm. To this end consider for $f_y $ for $ y \in K $,
  \begin{align*}
    &\frac{|A^n f_y(\lambda) -A^mf_y(\lambda)|}{\rho(\lambda)}\\
    &\quad =\frac{f_y(\lambda) \langle g, \lambda \rangle }{\rho(\lambda)}|\int (\exp(\langle y, \nu \xi \rangle) -1 -\langle y, \nu \xi \rangle) (\mu^n(d\xi)- \mu^m(d\xi))|\\
    &\quad =\frac{f_y(\lambda) \langle g, \lambda \rangle }{\rho(\lambda)}|\int (\xi^2 \langle y, \nu  \rangle^2 + o(\langle y, \nu \rangle^2 \xi^2) ) (\mu^n(d\xi)- \mu^m(d\xi))|\\
    &\quad =\frac{f_y(\lambda) \langle g, \lambda \rangle }{\rho(\lambda)}\Big(\underbrace{|\int (\xi^2 \langle y, \nu \rangle^2 +  o(\langle y, \nu \rangle^2 \xi^2) )\sigma^2 (n^2 \delta_{\frac{1}{n}}(d\xi)-m^2 \delta_{\frac{1}{m}}(d\xi))|}_{a^1_{nm}}\\
    &\quad \quad + \underbrace{|\int
      (\xi^2 \langle y, \nu \rangle^2 +
      o(\langle y, \nu \rangle^2 \xi^2)
      )(1_{\{\xi \geq
      \frac{1}{n}\}}-1_{\{\xi \geq
      \frac{1}{m}\}}
      )m(d\xi)|}_{a^2_{nm}}\Big).
  \end{align*}
  From this we infer \eqref{eq:condIImod}, where
  $a_{nm}= C(a^1_{nm}+a^2_{nm})$ for some constant $C$, for \emph{all}
  $ y \in K $. The conditions of Theorem \ref{thm:approximation_gen}
  are therefore satisfied and we obtain a generalized Feller semigroup
  whose generator is given by \eqref{eq:genmain}.
  
  Concerning the second assumption, let us proceed similarly as in the
  proof of Proposition \ref{prop:epsjumps} (ii). Take
  $y \in \mathcal{D}$.  Consider the local martingale
  \begin{align}\label{eq:martmain}
    M^y_t= \langle y, \lambda_t \rangle - \langle y, \lambda_0 \rangle -\int_0^t \langle \mathcal{A} y, \lambda_s \rangle - \beta \langle y, \nu \rangle \langle g, \lambda_s \rangle ds. 
  \end{align}
  Let now $y$ be such that $\langle y, \nu \rangle =1$ and denote
  $M^y$ for such $y$ simply by $M$. By the carr\'e du champs formula
  the predictable quadratic variation\footnote{This is also denoted by
    angle brackets.} of $M$ is given by
  \[
    \langle M, M \rangle_t = \int_0^t (Af^2(\lambda_s)-2f(\lambda_s)
    Af(\lambda_s))ds=\left(\sigma^2 + \int \xi^2 m(d\xi)\right)
    \int_0^t \langle g, \lambda_s \rangle ds
  \]
  for $f(\lambda)=\langle y, \lambda \rangle$ and all
  $\langle y, \nu \rangle =1$, which shows in particular that $M$ is
  independent of the particular representative $y$.  Moreover, due to
  \eqref{eq:approxSPDE}
  \[
    \sum_{s \leq t} (\Delta M_s)^2= \lim_{n \to \infty}
    (\frac{1}{n^2}K^n_1 + \sum_{s \leq t} (\Delta
    N_{2,s}^n)^2)=\lim_{n \to \infty}\sum_{s \leq t} (\Delta
    N_{2,s}^n)^2
  \]
  where $K^n_1$ denotes the number of jumps of $N_1^n$ up to time $t$.
  Hence, the compensator of $\sum_{s \leq t} (\Delta M_s)^2$ is
  $\int_0^t \int \xi^2 m(d\xi) \langle g, \lambda_s \rangle ds$ and we
  deduce by the unique decomposition of $M$ into a continuous and
  purely discontinuous local martingale $M^c$ and $M^d$, that
  $\langle M^c, M^c \rangle = \sigma^2 \int_0^t \langle g, \lambda_s
  \rangle ds$. Therefore $M_t$ can be written as
  \[
    M_t= M_t^c+M^d_t =\int \sigma \langle g, \lambda_s \rangle dB_s +
    \int_0^t \int \xi (\mu^X(d\xi,ds) - \langle g, \lambda_s \rangle
    m(d\xi) ds),
  \]
  for some Brownian motion $B$ and random measure $\mu^X(d\xi,ds)
  $. Define now $X$ as
  \[
    X_t= \int_0^t \beta \langle g,\lambda_s \rangle ds +M_t.
  \]
  Then $X$ is of form \eqref{eq:X}. Furthermore note that $M^y$ of
  \eqref{eq:martmain} for $y \in \mathcal{D}$ is given by
  $M^y =\langle y, \nu \rangle M$ which can be seen again from the
  carr\'e du champs formula, yielding
  $\langle M^y, M^y \rangle = \langle y, \nu \rangle^2 \langle M, M
  \rangle$. We thus deduce from \eqref{eq:martmain}
  \[
    \langle y, \lambda_t \rangle=\langle y, \lambda_0 \rangle+
    \int_0^t \langle \mathcal{A} y, \lambda_s \rangle ds + \langle y,
    \nu \rangle X_t
  \]
  for $ y \in \mathcal{D} $.
  
  To pass from this weak formulation to the mild one, notice the
  following uniqueness assertion: fix the driving process $ X $ with
  c\`adl\`ag trajectories and $ {(\varphi_t(y))}_{t \geq 0} $, for
  $ y \in Y $, be another solution process with c\`ag trajectories
  depending linearly on $ y $.  We assume that
  \begin{align}\label{eq:varphi}
    \varphi_t(y)=\langle y, \lambda_0 \rangle+ \int_0^t \varphi_s(\mathcal{A} y)
    ds + \langle y, \nu \rangle X_t
  \end{align}
  for all $ y \in \mathcal{D} $ and furthermore for every sequence
  $ y_n \to y $ in $ Y $ we can choose a set of probability one where
  for every $ s \geq 0 $
  \begin{align}\label{eq:contphi}
    \varphi_s(y_n) \to \varphi_s(y)
  \end{align}
  as $ n \to \infty $ holds true. Note that the set will depend on the
  chosen sequence. Then the announced uniqueness assertion means that
  necessarily $ \varphi_t(y) = \langle y , \lambda_t \rangle $ for all
  $ y \in \mathcal{D}$. Indeed, the difference
  $ g_t(y) = \varphi_t(y) - \langle y , \lambda_t \rangle $ satisfies
  \[
    g_t(y) = \int_0^t g_s(\mathcal{A}y) ds
  \]
  with c\`ag trajectories. Therefore clearly
  $ s \mapsto g_{t-s}(S_sy) $ is constant simply by taking the
  derivative with respect to $ s $ and collecting terms for
  $ y \in \mathcal{D} $, hence
  $ 0 = g_0(S_ty) = g_t(y) = \varphi_t(y) - \langle y , \lambda_t
  \rangle $.

  The mild formulation then follows from this uniqueness result
  immediately. Indeed, define processes
  \[
    f_t(y) := \langle \mathcal{S}_t y, \lambda_0 \rangle +\int_0^t
    \langle \mathcal{S}_{t-s} y , \nu \rangle dX_s,
  \]
  for all $ y \in Y $ and $ t \geq 0 $.  For $ y \in \mathcal{D} $ we
  then have
  \begin{align*}
    & \langle  y, \lambda_0 \rangle+ \int_0^t f_s(\mathcal{A} y) ds + \langle y, \nu \rangle X_t \\ 
    &\quad =  \langle  y, \lambda_0 \rangle +
      \int_0^t \langle \mathcal{A} \mathcal{S}_{s} y, \lambda_0 \rangle ds +  \int_0^t \int_0^s \langle \mathcal{A} \mathcal{S}_{s-u} y, \nu \rangle dX_u ds + \langle y, \nu 
      \rangle X_t   \\
    &\quad =  \langle \mathcal{S}_t y , \lambda_0 \rangle + \int_0^t \int_u^t \langle \mathcal{A} \mathcal{S}_{s-u} y, \nu \rangle ds dX_u  + \langle y, \nu 
      \rangle X_t  \\
    &\quad  = \langle \mathcal{S}_t y , \lambda_0 \rangle + \int_0^t \langle \mathcal{S}_{t-u} y - y, \nu \rangle dX_u  + \langle y, \nu 
      \rangle X_t \\
    &\quad  = \langle \mathcal{S}_t y , \lambda_0 \rangle + \int_0^t \langle \mathcal{S}_{t-u} y, \nu \rangle dX_u \\
    &\quad  =  f_t(y) \, ,
  \end{align*}
  whence $ {(f_t(y))}_{t \geq 0} $ satisfies \eqref{eq:varphi}. Since
  $(f_t(y))$ is linear in $y$ and satisfies by basic dominated
  convergence for stochastic integrals the continuity assumption
  \eqref{eq:contphi}, the previous uniqueness result yields
  $ f_t(y) = \langle y , \lambda_t \rangle $ for all $ t \geq 0 $
  almost surely.

  The third assertion can be shown exactly along the lines of the
  proof of the second assertion of Proposition \ref{prop:jumpseps0}.

  The fourth assertion is a simple consequence of (ii) with $ y = g $.
    
  Finally to prove (v), set in (iii) $y_0 =u g$. Then, the mild
  solution of $y_t$ is given by
  \begin{align}\label{eq:VoltRic}
    y_t = u \mathcal{S}_t  g+ \int_0^t \mathcal{S}_{t-s} g\mathcal{R}(\langle y_s, \nu \rangle) ds.
  \end{align}
  Hence, by (iii) and Fubini's theorem
  \begin{align*}
    \mathbb{E}_{\lambda_0}[ \exp(u V_t)]&= \mathbb{E}_{\lambda_0}[ \exp(u \langle g, \lambda_t \rangle)]=\exp(\langle y_t, \lambda_0 \rangle)\\& = 
                                                                                                                                                 \exp(u \langle g, \mathcal{S}_t^* \lambda_0 \rangle+ \int_0^t \langle g, \mathcal{S}^*_{t-s} \lambda_0 \rangle \mathcal{R}(\langle y_s, \nu \rangle) ds).
  \end{align*}
  Setting now $\psi_t:= \langle y_t, \nu \rangle$ for $t >0$, we see
  from \eqref{eq:VoltRic}
  \[
    \psi_t= u \langle g, \mathcal{S}_t^* \nu \rangle + \int_0^t\langle
    g, \mathcal{S}^*_{t-s} \nu\rangle\mathcal{R}(\psi_s) ds= u
    K(t)+\int_0^t K(t-s) \mathcal{R}(\psi_s)ds,
  \]
  and in particular that it is well-defined for $t >0$.  This proves
  the assertion of uniqueness in law, too.
\end{proof}

\section{Concrete specifications}\label{sec:concrete}

The goal of this section is to concretely specify the abstract setting
introduced in Section \ref{sec:markovianlift_abstract} when we deal
with specific kernels of form \eqref{eq:kernel}.

\subsection{Lifting rough Volterra processes to
  measures}\label{sec:markovianlift}

As already outlined in the introduction, we here consider the case of
kernels given as Laplace transforms of some (signed) measure $\nu$ on
$\mathbb{R}_+$, i.e.
\[
  K(t)=\int_0^{\infty} e^{-xt} \nu(dx),
\]
such that $K(t) < \infty$ for all $t >0$ and
$K \in L^2_{\text{loc}}(\mathbb{R}_+, \mathbb{R})$. If $\nu$ is a
nonnegative measure, we are in the important setting of
\emph{completely monotone kernels}, into which for instance the
fractional kernels appearing in rough volatility models of the form
$K(t)=\Gamma(\alpha)^{-1}t^{\alpha-1}$ for
$\alpha \in (\frac{1}{2},1)$ fall.

To cast this into the framework considered in Section
\ref{sec:markovianlift_abstract}, in particular Assumption
\ref{ass:weak_existence}, let $Y$ be the space of bounded continuous
functions on the extended real half-line
$\overline{\mathbb{R}}_+:= \mathbb{R}_{+} \cup \{\infty\}$. We here
compactify $\mathbb{R}_+$, in order to make
$C_b(\overline{\mathbb{R}}_+, \mathbb{R})$ separable.  Its dual space
$Y^*$ is the space of the finite (signed) regular Borel measures on
the extend real half-line $ \overline{\mathbb{R}}_+ $. As above we
consider the weak-$*$-topology on $Y^*$ and a weight function
\[
  \rho(\lambda) = 1 + {|\lambda|}^2(\overline{\mathbb{R}}_{+}) \, ,
\]
where $ | \lambda | $ denotes the total variation norm of $ \lambda$.
Moreover, $ Z $ is the space of functions $ g \in Y $ such that
\[
  (x \mapsto x g(x) ) \in Y
\]
together with the operator norm on it,
i.e.~$ \| g \| = \sqrt{\| g \|^2 + \sup_{x \geq 0} | xg(x) |^2} $ for
$ g \in Z $. Its dual $Z^*$ is the space of regular Borel measures
$\nu$ on $\overline{\mathbb{R}}_+$ that satisfy
\[
  \int_0^{\infty} (\frac{1}{x} \wedge 1) \, \nu(dx) <\infty.
\]
As semigroup $(\mathcal{S}^*_t)_{t \geq 0}$ acting on $Z^*$ and $Y^*$
we consider the multiplication semigroup
\[
  \mathcal{S}_t^* \lambda= e^{-t\cdot} \lambda
\]
which leaves $ Y^* $ and $ Z^* $ invariant and acts in a strongly
continuous way thereon with respect to the strong norms.  Moreover,
$ \lambda \mapsto \mathcal{S}^*_t\lambda $ is weak-$*$-continuous on
$ Y^* $ and on $ Z^* $ for every $ t \geq 0 $ (considering the
weak-$*$-topology on both the domain and the image space). Whence, as
stated in Remark \ref{rem:examplespaces}, this semigroup then
satisfies all requirements of Assumption \ref{ass:weak_existence} (iv)
to (vi).

In terms of \eqref{eq:kernel}, $K$ can now be written as
\[
  K(t)=\int_0^\infty e^{-xt} \nu(dx)=\langle 1, \mathcal{S}_t^* \nu
  \rangle,
\]
where the pairing is here
$\langle y, \lambda \rangle= \int_0^{\infty} y(x) \lambda(dx)$. The
element $g \in Y$ appearing in \eqref{eq:kernel} is here simply the
constant function $1$. Note also that the requirements of Assumption
\ref{ass:semigroup} are satisfied as well. Indeed
$\mathcal{S}^*_t \nu \in Y^* $ for all $ t > 0 $ and
\[
  \int_0^t \| \mathcal{S}^*_s \nu \|^2_{Y^*} ds = \int_0^t
  (\int_0^{\infty} e^{-xs} \nu(dx))^2 ds = \int_0^t K(s)^2 ds < \infty
\]
for all $ t > 0 $, since $K$ is assumed to lie in
$L^2_{\text{loc}}(\mathbb{R}_+, \mathbb{R})$.
 
As in \eqref{eq:statespaceEw}, we define for fixed $ w > 0 $
\begin{equation}
  \begin{split}\label{eq:statespaceEwlambda}
    \mathcal{E}^w:=& \{ \lambda_0 \in Y^* |  \langle 1, \mathcal{S}_t^{*} \lambda_0 \rangle -  \int_0^t R^w(t-s) \langle 1, \mathcal{S}_s^{*} \lambda_0\rangle ds \geq 0 \text{ for all }  t \geq 0  \}   \\
    = & \{ \lambda_0 \in Y^* | \overline{\lambda}_t \geq 0
    \text{ with }\\
    &\quad \quad \quad \quad \lambda_t(dx) = e^{-xt} \lambda_0(dx) - w
    \int_0^t e^{-(t-s)x} \nu(dx) \overline{\lambda}_s ds \text{ for
      all } t \geq 0 \},
  \end{split}
\end{equation}
where $R^w$ denotes the resolvent of $ w K(t)$ and
$\overline{\lambda}=\langle 1,\lambda\rangle$ the total mass of
$\lambda$.

\begin{remark}\label{ex:local_compactness}
  Notice that the cone $ \mathcal{E}^w $ is \emph{not} locally compact
  with respect to the weak-$*$-topology. Indeed, consider the simplest
  case $ K = 1 $ and denote the Laplace of some signed measure
  $\lambda_0$ by
  \[
    h(t) = \int_0^\infty \exp(-xt) \lambda_0(dx).
  \]
  Then $\lambda_0$ lies in $ \mathcal{E}^w $ if $ h(t) \geq 0 $ and
  $ h'(t) \geq 0 $ for all $ t \geq 0 $. Choose a positive measure
  $ \lambda_0 $ whose support does \emph{not} contain $ [1,2] $. Then
  the signed measures
  $ \lambda_n = \lambda_0 (dx) + 1_{[1,2]}\frac{1}{x} \sin(n \pi x) dx
  \in \mathcal{E}^w $, for $ n \geq 1 $, and
  $ \lambda_n \to \lambda_0 $ in the weak-$*$-topology. On the other
  hand the total variation norm of $ \lambda_n $ does \emph{not}
  converge to the total variation norm of $ \lambda_0 $. Hence,
  $\lambda \mapsto 1+ |\lambda|^2(\overline{\mathbb{R}}_+)$ is not
  continuous and by Remark \ref{rem:local_compactness} the cone
  $ \mathcal{E}^w $ is not locally compact.
  
  In contrast, note that the cone of non-negative Borel measures on
  $\overline{\mathbb{R}}_+$ (or generally on a compactum) is locally
  compact in the weak-$*$-topology. In this case the total variation
  norm of a non-negative measures is nothing but the total mass,
  i.e.~a linear functional on it, thus continuous. The one-point
  compactification in this case is called Watanabe topology, see
  \cite{daw:12}.

\end{remark}

Assumption \ref{ass:crucial_abstract} now reads as follows:

\begin{assumption}\label{ass:crucial}
  Let $ \nu $ be such that $ \exp(-t.)\nu \in \mathcal{E}^w $ for all
  $ t > 0 $ and for all $ w > 0 $.
\end{assumption}

\begin{remark}
  As stated in Remark \ref{rem:crucial_abstract}, Assumption
  \ref{ass:crucial} is satisfied if $ R^w \geq 0 $ for all $w>0$ and
  $K \geq 0$, which is satisfied if $K$ is completely monotone.
\end{remark}

The state space that we consider for the stochastic processes in the
sequel is as of \eqref{statespaceE}, namely
\[
  \mathcal{E} = \cap_{w > 0} \mathcal{E}^w \, .
\]

\begin{remark}
  Comparing $\mathcal{E}$ with the state spaces considered in
  \cite{AbiElE:18b} (e.g. Equation (4.6)), we see that the conditions
  there translate to
  \[
    \{ \lambda_0 \in Y^* \, | \, t \mapsto \langle 1,
    \mathcal{S}_t^*\lambda_0 \rangle= \int_0^{\infty} e^{-tx}
    \lambda_0(dx) \in \mathcal{G}_K\},
  \]
  where $\mathcal{G}_K$ is defined in \cite[Equation
  (2.5)]{AbiElE:18b}. From \cite[Theorem A.2]{AbiElE:18b} we see that
  this describes a subspace of $\mathcal{E} $ in case a resolvent of
  the first kind exists for the kernel $ K $.
\end{remark}

It is interesting to consider the case when $ \nu $ is a measure with
finitely many atoms: we identify this measure setting then with some
$ \mathbb{R}^N $ and denote
$ \operatorname{supp}(\nu) = \{ 0 \leq x_1 < \cdots < x_N \} $.
\begin{proposition}
  Under Assumption \ref{ass:crucial} the state space $ \mathcal{E} $
  is a well defined closed cone given as the intersection of
  $\cap_{w >0} \mathcal{E}^w $ where for $w >0$
  \[
    \mathcal{E}^w=\{ \lambda_0 \in \mathbb{R}^N \, |\, \langle e^{t
      A^w} \lambda_0 , \mathbf{1} \rangle_{\mathbb{R}^N} \geq 0 \quad
    \text{for all } t \geq 0 \}
  \]
  with
  $A^w= \diag(-x_1, \ldots, -x_N)- w (\nu(x_1) \mathbf{1}, \ldots,
  \nu(x_N) \mathbf{1})^{\top}$. Here, $\mathbf{1}\in \mathbb{R}^N$ is
  the (column) vector consisting of ones.
\end{proposition}

\begin{proof}
  Note that the deterministic equation as of
  \eqref{eq:lambda_abstract} has here the following form
  \[
    d\lambda_t(x_i)=-x_i\lambda_t(x_i) dt- w \nu(x_i)\sum_{j=1}^N
    \lambda_t(x_j) dt, \quad i \in \{1, \ldots,N\}.
  \]
  This can be written as $d\lambda_t=A^w \lambda_t dt$ so that the
  assertion follows from the definition of $\mathcal{E}^w$.
\end{proof}

\begin{example}
  There are simple two dimensional examples when
  $ \mathcal{E}^0 \neq \mathcal{E}^w $ for $ w > 0 $: indeed take
  $ x_0 = 0 $ and $ x_1 = 1 $ and
  $ \nu(dx) = \delta_0(dx) + \delta_1(dx) $. In this case
  \[
    \mathcal{E}^0 =\operatorname{cone}(e_2, e_1-e_2)
  \]
  whereas
  \[
    \mathcal{E}^w = \operatorname{cone} (e_{1}+e_{2} , e_{1} - e_{2})
    \, .
  \]
  Here $e_i$, $i=1,2$ denotes the canonical basis vectors of
  $\mathbb{R}^2$.  Observe here that $ \mathcal{E}^w $ does not depend
  on $ w $ for $ w > 0 $, but it changes when $ w \to 0 $.
\end{example}

Let us finally come to the SPDE formulation which will in particular
lead to rough affine volatility models. Consider the following measure
valued SPDE
\begin{align}\label{eq:SPDEmain1}
  d \lambda_t(dx) = -x \lambda_t(dx) dt + \nu(dx) dX_t,
\end{align}
where $(X_t)_{t \geq 0}$ is an It\^o semimartingale of the form
\begin{align}\label{eq:X1}
  X_t = \int_0^t \beta \overline{\lambda}_s ds+ \int_0^t
  \sigma \sqrt{\overline{\lambda}_t}dB_t+ \int_0^t \int \xi
  (\mu^X(d\xi,ds) - \overline{ \lambda}_s  m(d\xi) ds),
\end{align}
for some Brownian motion $B$ and random jump measure $\mu^X$ with
$\beta, \sigma \in \mathbb{R}$ and $m(d\xi)$ is a L\'evy measure on
$\mathbb{R}_{+}$ admitting a second moment. Recall here that
$\overline {\lambda}=\langle 1, \lambda \rangle$. As a corollary of
Theorem \ref{thm:main} we now obtain the following result.

\begin{theorem}
  Let Assumption \ref{ass:crucial} be in force.
  \begin{enumerate}
  \item Then the stochastic partial differential equation
    \eqref{eq:SPDEmain1} admits a unique Markovian solution
    $(\lambda_t)_{t\geq0} $ in $ \mathcal{E} $ given by a generalized
    Feller semigroup on $ \mathcal{B}^\rho(\mathcal{E}) $ whose
    generator takes the form
    \begin{align}\label{eq:genmain1}
      Af_y(\lambda)=f_y(\lambda) (\int_0^{\infty}-x y(x)\lambda (dx)+ \mathcal{R}(\langle y, \nu \rangle) \overline{\lambda} 
    \end{align}
    on the set of Fourier basis elements $f_y$ with
    $y \in \mathcal{D} \cap \mathcal{E}_*$, where
    $\mathcal{R}: \mathbb{R}_- \to \mathbb{R}$ is given by
    \begin{align}\label{eq:Rtrad1}
      \mathcal{R}(u)=\beta u+ \frac{1}{2} \sigma^2 u^2+\int_{\mathbb{R}_+}
      \left(\exp( u \xi )-1 - u\xi \right) m(d\xi).
    \end{align}
  \item This generalized Feller process allows to construct a
    probabilistically weak and analytically weakly mild solution of
    \eqref{eq:SPDEmain1}, i.e.~for $y \in Y$
    \[
      \langle y, \lambda_t \rangle = \langle e^{-t\cdot} y,
      \lambda_0\rangle +\int_0^t \langle e^{-(t-s)\cdot}y , \nu
      \rangle dX_s,
    \]
    in particular for every initial value the semimartingale $X$ can
    be constructed on an appropriate probabilistic basis.  For every
    $ y \in Y $ there exists a c\`ag version of the real valued
    process $ {(\langle y , \lambda_t \rangle )}_{t \geq 0} $.
    
  \item The affine transform formula is satisfied, i.e.
    \[
      \mathbb{E}_{\lambda_0}\left[ \exp( \langle y_0, \lambda_t
        \rangle)\right]=\exp(\langle y_t, \lambda_0 \rangle),
    \]
    where $y_t$ solves $\partial_t y_t=R(y_t)$ for all
    $y_0 \in \mathcal{E}_*$ and $t >0$ in the mild sense with
    $R: \mathcal{D} \cap \mathcal{E}_* \to \mathbb{R}$ given by
    \[
      R(y)(x) = -x y(x) + \mathcal{R}( \langle y , \nu \rangle)
    \]
    and $\mathcal{R}$ defined in \eqref{eq:Rtrad1}.  Furthermore,
    $y_t \in \mathcal{E}_*$ for all $t \geq 0$.
  \item For all $\lambda_0 \in \mathcal{E}$, the corresponding jump
    diffusion stochastic Volterra equation,
    $V_t = \overline{\lambda}_t=\langle 1, \lambda_t \rangle$, given
    by
    \begin{align}\label{eq:Voltrep1}
      \overline{\lambda}_t= \int_0^{\infty} e^{-tx} \lambda_0(dx)+
      \int_0^t K(t-s)dX_s
    \end{align}
    with $X$ given by \eqref{eq:X1} admits a unique probabilistically
    weak solution in $\mathbb{R}_+$ with c\`ag trajectories.
  \item[(v)] The Laplace transform of the Volterra equation $V_t$ is
    given by
    \begin{align*}
      \mathbb{E}_{\lambda_0}\left[\exp(u V_t)\right]=\exp(u h(t)+ \int_0^t h(t-s) \mathcal{R}(\psi_s) ds),
    \end{align*}
    where $h(t)=\int_0^{\infty} e^{-xt} \lambda_0(dx)$ and
    $\psi_t:=\langle y_t, \nu\rangle$ solves the Riccati Volterra
    equation
    \[
      \psi_t=u K(t)+\int_0^t K(t-s) \mathcal{R}(\psi_s) ds, \quad t >0
      \, .
    \]
    Hence the solution of the stochastic Volterra equation in
    \eqref{eq:Voltrep1} is unique in law.
  \end{enumerate}
\end{theorem}

\subsection{Lifting rough Volterra processes to forward
  curves}\label{sec:markovianlift_forward}

There are several ways to lift stochastic Volterra processes to
Markovian processes: in the previous section the focus was on the
completely monotone nature of many Volterra kernels. The lift that we
treat here considers instead the stochastic Volterra process together
with its conditional expectations from future values. As in the
Heath-Jarrow-Morton (HJM) case, this yields shift semigroups. This
Markovian lift also falls in the realm of Section
\ref{sec:markovianlift_abstract} and can be considered from the
perspective of generalized Feller semigroups.  Indeed, here $K$
(satisfying Conditions \ref{ass:K} below) is represented by
\[
  K(t)= \int_0^{\infty} \mathcal{S}^{\ast}_t K(x) \delta_0(dx),
\]
where $\mathcal{S}^*_t$ denotes the shift semigroup. In order meet all
the conditions of Section 4, we here assume the following:
\begin{assumption}\label{ass:K}
  \begin{enumerate}
  \item $K$ is assumed to lie in
    $L^2_{\text{loc}}(\mathbb{R}_+, \mathbb{R})$.
  \item there exists $L \in AC(\mathbb{R}_{+}, \mathbb{R})$, where
    $AC(\mathbb{R}_{+}, \mathbb{R})$ denotes the space of real-valued
    absolutely continuous functions on $\mathbb{R}_{+}$, such that
    $ K= L' $.
  \item There exists some strictly positive weight function $ v > 0 $
    such that
    $$ \int_0^{\infty} |K(x)|^2 v(x) dx = \int_0^{\infty} |L'(x)|^2 v(x) dx < \infty \, .$$
  \end{enumerate}
\end{assumption}

In terms of the underlying Banach spaces we we let $Y$ and in turn
also $Y^*$ be Hilbert spaces, similarly as in Filipovic \cite{F:01},
namely
\[
  Y=\{ y \in AC(\mathbb{R}_+, \mathbb{R}) \, | \, \int_0^{\infty}
  |y'(x)|^2 v(x) dx < \infty\}
\]
for the specified strictly positive weight function $ v > 0 $. We
endow $Y$ with the scalar product
\[
  \langle y, \lambda \rangle_{v}= y(0)\lambda(0) +\int y'(x)
  \lambda'(x) v(x) dx
\]
and denote the associated norm via $\| \cdot \|_{v}$.
 
\begin{assumption}\label{ass:stronglycontshift}
  We furthermore assume that the shift semigroup $ \mathcal{S}^* $
  acts in a strongly continuous way on $ Y^*$, which can be easily
  expressed in terms of the function $ v $, and that
  $\mathcal{S}^*_t K \in Y^*$ for $ t > 0 $.
\end{assumption}

\begin{remark}
  One can choose $ v(x) = \exp(\alpha x) $ for $ x \geq 0 $, which
  proved to be useful in case of term structures of interest rates.
\end{remark}

The weight function, in the sense of weighted spaces, is again given
by
\[
  \rho(\lambda) = 1 + \| \lambda \|_{v}^2 .
\]

Moreover, for $ Z $ we consider the subspace of functions vanishing at
$ 0 $ and whose first (weak) derivative lies in $ Y $ (with the
corresponding operator Hilbert space norm), then $ K \in Z^* $ can be
viewed as a weak derivative of a function $ L \in Y^* $ with norm
\[ {\|K\|}^2_{Z^*} := \int_0^\infty {|L'(x)|}^2 v(x) dx.
\]
In particular the value $ \infty $ at $ 0 $ is possible. By condition
(ii) and (iii) of Assumption \ref{ass:K}, $K \in Z^*$. Notice also the
following inclusions
\[
  Z \subset Y = Y^* \subset Z^* \, .
\]
In terms of \eqref{eq:kernel}, $K \in Z^*$ can now be written as
\[
  K(t)=\langle \mathcal{S}_t^* K,1 \rangle_{v}= ({S}_t^* K)(0).
\]
The element $g \in Y$ appearing in \eqref{eq:kernel} is here again
simply the constant function $1$.

Note that due to Assumption \ref{ass:stronglycontshift} the shift
semigroup $(\mathcal{S}^*_t)_{t \geq 0}$ acting on $Z^*$ and $Y^*$
satisfies all requirements of Assumption \ref{ass:weak_existence} (iv)
to (v).  Concerning (vi), the (pre)adjoint
$(\mathcal{S}_t)_{t \geq 0}$ acts on $Z$ and $Y$ via the following
formula obtained from adjoining with respect to the scalar product on
$Y$
\[
  \mathcal{S}_t y(x)=y(0)+ \int_0^{x \wedge t}
  y(0)\frac{1}{v(s)}ds+\int_{t}^{x \vee t}
  y'(s-t)\frac{v(s-t)}{v(s)}ds.
\]

The generator $A$ can be calculated accordingly by taking the
derivative with respect to $t$. Notice that the pre-dual semigroup
$ \mathcal{S} $ acts in strongly continuous way on $ Y $ (as dual
semigroup of a strongly continuous semigroup on a Hilbert space $Y$,
see \cite{engnag:00} for details) as well as on $ Z $. Furthermore the
requirements of Assumption \ref{ass:semigroup} are met as well.

As in \eqref{eq:statespaceEw}, we define for fixed $ w > 0 $
\begin{equation}
  \begin{split}
    \mathcal{E}^w:=& \{ \lambda_0 \in Y^* |  \langle 1, \mathcal{S}_t^{*} \lambda_0 \rangle -  \int_0^t R^w(t-s) \langle 1, \mathcal{S}_s^{*} \lambda_0\rangle ds \geq 0 \text{ for all }  t \geq 0  \}   \\
    = & \{ \lambda_0 \in Y^* | \lambda_t(0) \geq 0
    \text{ with }\\
    &\quad \quad \quad \quad \lambda_t(x) = \lambda_0(x+t) - w
    \int_0^t K(t-s+x) \lambda_s(0) ds \text{ for all } t \geq 0 \},
  \end{split}
\end{equation}
where $R^w$ denotes the resolvent of $ w K(t)$. Note that
$\langle 1,\lambda\rangle = \lambda(0)$.  Again the state space will
be.
\[
  \mathcal{E} = \cap_{w > 0} \mathcal{E}^w \, .
\]
 
Assumption \ref{ass:crucial_abstract} now reads as follows:

\begin{assumption}\label{ass:crucial1}
  Let $ K$ be such that $ K(t+\cdot) \in \mathcal{E}^w $ for all
  $ t > 0 $ and for all $ w > 0 $.
\end{assumption}

The SPDE that we consider in the present setting is
\begin{align}\label{eq:SPDEmain2}
  d \lambda_t(x) = \frac{d}{dx}\lambda_t(x) dt + K(x)dX_t,
\end{align}
where $(X_t)_{t \geq 0}$ is an It\^o semimartingale of the form
\begin{align}\label{eq:X2}
  X_t = \int_0^t \beta \lambda_s(0) ds+ \int_0^t
  \sigma \sqrt{\lambda_s(0)}dB_t+ \int_0^t \int \xi
  (\mu^X(d\xi,ds) - \lambda_s(0)  m(d\xi) ds),
\end{align}
for some Brownian motion $B$ and random jump measure $\mu^X$ with
$\beta, \sigma \in \mathbb{R}$ and $m(d\xi)$ is a L\'evy measure
admitting a second moment.  As a corollary of Theorem \ref{thm:main}
we now obtain the following result.

\begin{theorem}
  Let Assumptions \ref{ass:K}, \ref{ass:stronglycontshift} and
  \ref{ass:crucial1} be in force.
  \begin{enumerate}
  \item Then the stochastic partial differential equation
    \eqref{eq:SPDEmain2} admits a unique Markovian solution
    $(\lambda_t)_{t\geq0} $ in $ \mathcal{E} $ given by a generalized
    Feller semigroup on $ \mathcal{B}^\rho(\mathcal{E}) $ whose
    generator takes the form
    \begin{align*}
      Af_y(\lambda)=f_y(\lambda) (\langle \mathcal{A} y, \lambda \rangle_{v}
      + \mathcal{R}(\langle y, K \rangle_{v}) \lambda(0)
    \end{align*}
    on the set of Fourier basis elements $f_y$ with
    $y \in \mathcal{D} \cap \mathcal{E}_*$ where
    $\mathcal{R}: \mathbb{R}_- \to \mathbb{R}$ is given by
    \begin{align*}
      \mathcal{R}(u)=\beta u+ \frac{1}{2} \sigma^2 u^2+\int_{\mathbb{R}_+}
      \left(\exp( u \xi )-1 - u\xi \right) m(d\xi).
    \end{align*}
  \item This generalized Feller process allows to construct a
    probabilistically weak and analytically weakly mild c\`ag solution
    of \eqref{eq:SPDEmain2}, i.e.~in particular the point evaluations
    (being linear functionals) satisfy
    \[
      \lambda_t(x) = \lambda_0(t+x) +\int_0^t K(t-s+x) dX_s,
    \]
    are c\`ag, and for every initial value the semimartingale $X$ can
    be constructed on an appropriate probabilistic basis.
  \item The affine transform formula is satisfied, i.e.
    \[
      \mathbb{E}_{\lambda_0}\left[ \exp( {\langle y_0, \lambda_t
          \rangle}_v)\right]=\exp({\langle y_t, \lambda_0 \rangle}_v),
    \]
    where $y_t$ solves $\partial_t y_t=R(y_t)$ for all
    $y_0 \in \mathcal{E}_*$ and $t >0$ in the mild sense with
    $R: \mathcal{D} \cap \mathcal{E}_* \to \mathbb{R}$ given by
    \[
      R(y)(x) = \mathcal{A} y(x) + \mathcal{R}( \langle y , \nu
      \rangle_{v})
    \]
    with $\mathcal{R}$ defined in \eqref{eq:Rtrad}.  Furthermore,
    $y_t \in \mathcal{E}_*$ for all $t \geq 0$.
  \item For all $\lambda_0 \in \mathcal{E}$, the corresponding jump
    diffusion stochastic Volterra equation, i.e.
    \begin{align}\label{eq:Voltrep2}
      \lambda_t(0)= \lambda_0(t)+ \int_0^t K(t-s)dX_s
    \end{align}
    admits a probabilistically weak solution in $\mathbb{R}_+$ with
    c\`ag trajectories.
  \item[(v)] The Laplace transform of the Volterra equation $V_t$ is
    given by
    \begin{align*}
      \mathbb{E}_{\lambda_0}\left[\exp(u \lambda_t(0))\right]=\exp(u h(t)+ \int_0^t h(t-s) \mathcal{R}(\psi_s) ds),
    \end{align*}
    where $h(t)=\lambda_0(t)$ and $\psi_t:=\langle y_t, \nu\rangle$
    solves the Riccati Volterra equation
    \[
      \psi_t=u K(t)+\int_0^t K(t-s) \mathcal{R}(\psi_s) ds, \quad t >0
      \, .
    \]
    Hence the solution of the stochastic Volterra equation in
    \eqref{eq:Voltrep2} is unique in law.
  \end{enumerate}
\end{theorem}

\end{document}